\documentclass[11pt,reqno]{amsart}

\usepackage[a4paper,margin=30mm]{geometry}
\usepackage{amsmath,amssymb,amsthm,mathtools}
\numberwithin{equation}{section}
\usepackage{booktabs,longtable,array,graphicx,tabularx}
\providecommand{\LTcaptype}{table}
\usepackage{float,needspace}
\graphicspath{{figures/}}
\usepackage{enumitem}
\usepackage{xcolor}
\usepackage{hyperref}
\usepackage[nameinlink,capitalise,noabbrev]{cleveref}
\usepackage{microtype}
\usepackage{url}

\hypersetup{
  colorlinks=true,
  linkcolor=blue!55!black,
  citecolor=blue!55!black,
  urlcolor=blue!55!black,
  pdftitle={Local combinatorial criteria for geometric hyper-ideal triangulations via combinatorial Ricci flow},
  pdfauthor={Soichiro Uemura}
}

\allowdisplaybreaks
\makeatletter
\g@addto@macro\@adjustvertspacing{%
  \abovedisplayskip=4pt plus 2pt minus 1pt
  \belowdisplayskip=4pt plus 2pt minus 1pt
  \abovedisplayshortskip=0pt plus 1pt
  \belowdisplayshortskip=2pt plus 1pt minus 1pt
}
\makeatother
\setlist{nosep}
\usepackage{aliascnt}
\providecommand{\RevisionHeadingColor}{}
\newtheoremstyle{plainrefs}{.5\baselineskip plus .2\baselineskip}{.5\baselineskip plus .2\baselineskip}
  {\itshape}{}{\bfseries}{.}{.5em}
  {{\RevisionHeadingColor\thmname{#1}\thmnumber{ #2}\thmnote{ \normalfont#3}}}
\newtheoremstyle{definitionrefs}{.5\baselineskip plus .2\baselineskip}{.5\baselineskip plus .2\baselineskip}
  {\normalfont}{}{\bfseries}{.}{.5em}
  {{\RevisionHeadingColor\thmname{#1}\thmnumber{ #2}\thmnote{ \normalfont#3}}}
\newtheoremstyle{remarkrefs}{.5\baselineskip plus .2\baselineskip}{.5\baselineskip plus .2\baselineskip}
  {\normalfont}{}{\itshape}{.}{.5em}
  {{\RevisionHeadingColor\thmname{#1}\thmnumber{ #2}\thmnote{ \normalfont#3}}}
\theoremstyle{plainrefs}
\newtheorem{theorem}{Theorem}[section]
\theoremstyle{plainrefs}
\newaliascnt{proposition}{theorem}
\newtheorem{proposition}[proposition]{Proposition}
\aliascntresetthe{proposition}
\crefname{proposition}{Proposition}{Propositions}
\Crefname{proposition}{Proposition}{Propositions}
\theoremstyle{plainrefs}
\newaliascnt{lemma}{theorem}
\newtheorem{lemma}[lemma]{Lemma}
\aliascntresetthe{lemma}
\crefname{lemma}{Lemma}{Lemmas}
\Crefname{lemma}{Lemma}{Lemmas}
\theoremstyle{plainrefs}
\newaliascnt{corollary}{theorem}
\newtheorem{corollary}[corollary]{Corollary}
\aliascntresetthe{corollary}
\crefname{corollary}{Corollary}{Corollaries}
\Crefname{corollary}{Corollary}{Corollaries}
\theoremstyle{definitionrefs}
\newaliascnt{definition}{theorem}
\newtheorem{definition}[definition]{Definition}
\aliascntresetthe{definition}
\crefname{definition}{Definition}{Definitions}
\Crefname{definition}{Definition}{Definitions}
\theoremstyle{definitionrefs}
\newaliascnt{condition}{theorem}
\newtheorem{condition}[condition]{Condition}
\aliascntresetthe{condition}
\crefname{condition}{Condition}{Conditions}
\Crefname{condition}{Condition}{Conditions}
\theoremstyle{definitionrefs}
\newaliascnt{example}{theorem}
\newtheorem{example}[example]{Example}
\aliascntresetthe{example}
\crefname{example}{Example}{Examples}
\Crefname{example}{Example}{Examples}
\theoremstyle{remarkrefs}
\newaliascnt{remark}{theorem}
\newtheorem{remark}[remark]{Remark}
\aliascntresetthe{remark}
\newtheorem*{unnumberedremark}{Remark}
\crefname{remark}{Remark}{Remarks}
\Crefname{remark}{Remark}{Remarks}

\theoremstyle{plainrefs}
\newtheorem*{introbichromatic}{Theorem~\ref{thm:local-bichromatic}}
\newtheorem*{introparametric}{Theorem~\ref{thm:parametric-box}}

\newcommand{\R}{\mathbb{R}}

\newcommand{\PE}{P_E}
\newcommand{\Ktil}{\widetilde K}
\newcommand{\clamp}{\operatorname{clamp}}
\newcommand{\St}{\operatorname{St}}
\newcommand{\arccosh}{\operatorname{arccosh}}
\newcommand{\Vol}{\operatorname{Vol}}
\newcommand{\cov}{\operatorname{cov}}

\newcommand{\dd}{\,\mathrm d}

\newcommand{\un}{\mathrm{un}}
\newcommand{\Q}{\mathbb Q}
\newcommand{\Z}{\mathbb Z}
\newcommand{\code}[1]{\texttt{#1}}
\newcommand{\Pmix}{\mathcal P_{\mathrm{mix}}}
\title[Local criteria for geometric hyper-ideal triangulations]
{Local combinatorial criteria for geometric\\
hyper-ideal triangulations\\
via combinatorial Ricci flow}
\author{Soichiro Uemura}
\address{Kavli Institute for the Physics and Mathematics of the Universe (WPI), The University of Tokyo Institutes for Advanced Study, The University of Tokyo, Kashiwa, Chiba 277-8583, Japan}
\address{Graduate School of Mathematical Sciences, The University of Tokyo, Komaba, Meguro-ku, Tokyo 153-8914, Japan}
\address{RIKEN Center for Interdisciplinary Theoretical and Mathematical Sciences (iTHEMS), Hirosawa, Wako, Saitama 351-0198, Japan}
\email{soichiro.uemura@ipmu.jp or soichiro.uemura@gmail.com}
\date{}

\begin{document}

\begin{abstract}
We give local combinatorial criteria for realizing prescribed ideal triangulations by nondegenerate hyperbolic truncated tetrahedra. A local bichromatic criterion allows valence-7 edges when tetrahedra meeting low-valence edges use at most two quotient-edge classes. Using the extended combinatorial Ricci flow, we obtain parameter-dependent criteria for triangulations of minimum valence 6. Symmetric and pair-min angle estimates yield explicit valence conditions, including mixed conditions on entire edge stars. The proofs use analytic angle inequalities and rigorous interval evaluations. Explicit face pairings realize the criteria and distinguish their scope, while cyclic constructions and finite covers give infinite families. Each criterion yields a unique zero-curvature hyper-ideal metric and exponential convergence from any positive initial length vector.
\end{abstract}

\maketitle

\section{Introduction}

A central problem in the hyperbolic geometry of $3$-manifolds is to realize topological triangulations by nondegenerate hyperbolic tetrahedra. For complete finite-volume cusped hyperbolic $3$-manifolds, Thurston's geometric triangulation conjecture predicts a decomposition into positive-volume ideal hyperbolic tetrahedra; see \cite{ThurstonNotes,HodgsonRubinstein,HamPurcell}. For compact hyperbolic $3$-manifolds with totally geodesic boundary, the corresponding pieces are hyper-ideal, or truncated, tetrahedra. Kojima's theorem gives a canonical decomposition into hyperbolic truncated polyhedra \cite{Kojima}. This is an existence theorem for a polyhedral decomposition; it does not assert that an arbitrary prescribed ideal triangulation can be realized geometrically.

For an ideal triangulation of a compact $3$-manifold $N$ with nonempty boundary, a \emph{geometric realization} consists of genuine (nondegenerate) hyper-ideal tetrahedra glued by the prescribed face isometries, with the induced metric on $N$ smooth and hyperbolic with totally geodesic boundary. Once the face shapes match, smoothness is equivalent to total dihedral angle $2\pi$ around every quotient edge. The metric is complete on the compact space $N$, including its boundary; no cusp-completeness condition is needed. This is not a claim of completeness for the interior alone or for the space obtained by collapsing the boundary. See \Cref{def:geometric,prop:geometric-realization}.

Let the boundary components of $N$ have genus at least 2. Costantino--Frigerio--Martelli--Petronio proved that, if every edge of an ideal triangulation has valence at least 6, then $N$ admits a hyperbolic metric with totally geodesic boundary \cite{CFMP}. This does not establish geometricity of the prescribed triangulation. Feng--Ge--Hua proved geometricity under the uniform minimum-valence-10 hypothesis using the extended combinatorial Ricci flow \cite{FengGeHua}; Zhao lowered this threshold to 9 using the same flow \cite{Zhao}.

Our local criteria allow edges of valence 6, 7, or 8 under restrictions on quotient-edge identifications or neighbouring valences.

Write $E$ for the quotient-edge set and $v(e)$ for the valence of a quotient edge $e$, counting local occurrences. Our first main result uses local identifications rather than bounds on individual neighbouring valences.

\begin{introbichromatic}
Let $(M,\mathcal T)$ be a closed pseudo $3$-manifold with $v(e)\ge7$ for every $e\in E$. If every tetrahedron containing an occurrence of an edge of valence 7, 8, or 9 uses at most 2 quotient-edge classes, then $(M,\mathcal T)$ admits a unique zero-curvature hyper-ideal metric. The extended combinatorial Ricci flow converges exponentially to it from any positive initial length vector.
\end{introbichromatic}

Our second main result applies to triangulations of minimum valence 6. For a common parameter $C\in(3/2,1+\sqrt2]$ and integers $n\ge6$, put $c_n=\cos(2\pi/n)$ and
\begin{equation*}
 \Gamma_n(C)=\frac{C+2-c_n}{C+c_n},\qquad
 B_n(C)=
 \begin{cases}
 C,&6\le n\le9,\\[1mm]
 1+\dfrac{2C^2(1-c_n)}{1+c_n},&n\ge10.
 \end{cases}
\end{equation*}
These define a box $Q_C=\prod_{e\in E}[\Gamma_{v(e)}(C),B_{v(e)}(C)]$ in the variables $x_e=\cosh l_e$.

\begin{introparametric}
Let $(M,\mathcal T)$ be a closed pseudo $3$-manifold with $v(e)\ge6$ for every $e\in E$. Fix $C\in(3/2,1+\sqrt2]$. Suppose that, for every edge $e$ of valence 6, 7, 8, or 9 and every $x\in Q_C$ with $x_e=C$, the total incident extended angle at $e$ is larger than $2\pi$. Then every solution starting in the interior of $Q_C$ remains in $Q_C$. There is a unique zero-curvature hyper-ideal metric, and the extended combinatorial Ricci flow converges exponentially to it from any positive initial length vector.
\end{introparametric}

The general criterion in \Cref{cor:combined-star} yields the explicit valence conditions of \Cref{cor:C0-explicit}. Face-pairing examples distinguish their scope, and finite covers give infinite families.

\Cref{sec:preliminaries} recalls the geometric setting and flow theory, and \Cref{sec:compactness} proves the local bichromatic criterion. \Cref{sec:C2} derives explicit valence conditions from Zhao's $C=2$ conditional estimates \cite{Zhao}; \Cref{sec:parametric} extends the endpoint bounds to $3/2<C\le1+\sqrt2$ and develops symmetric and pair-min criteria. Each criterion is accompanied by face-pairing examples. \Cref{sec:manifold} gives the consequences for compact manifolds with boundary and finite covers. The appendices contain the rigorous numerical evaluations, face-pairing data, and comparisons for the mixed pair-min example.

\subsection*{Acknowledgements}
The author is grateful to Professor Masahito Yamazaki for helpful discussions. The author warmly thanks Xinrong Zhao for sharing his revised manuscript before its public release and for explaining its arguments. This work was supported by the RIKEN Junior Research Associate Program.

\noindent\textit{AI assistance.} The author used ChatGPT (OpenAI) as an aid for tasks including the search for explicit triangulations in \Cref{prop:bichromatic-family,prop:cyclic-bipyramid} and Examples~\ref{ex:U6}--\ref{ex:mixed-pair-eight}, figure preparation, code development for the rigorous numerical evaluations in Sections~\ref{sec:C2} and~\ref{sec:parametric}, reference checking, translation, and language editing. The author has verified the correctness of the AI-assisted examples, code, and figures and takes full responsibility for the content of this paper.

\section{Hyper-ideal metrics and the extended combinatorial Ricci flow}\label{sec:preliminaries}

\subsection{Closed pseudo \texorpdfstring{$3$}{3}-manifolds and ideal triangulations}

Let $\widehat{\mathcal T}=\bigsqcup_{i=1}^N\sigma_i$ be a finite disjoint union of combinatorial tetrahedra, realized as closed affine $3$-simplices with their Euclidean subspace topology. Their faces carry the induced topology. Pair their codimension-one faces by the affine extensions of vertex bijections, and denote the quotient complex by $\mathcal T$ and its underlying space, endowed with the quotient topology, by $M$.  The pair $(M,\mathcal T)$ is a closed pseudo $3$-manifold.  Let $E=E(\mathcal T)$ be its set of quotient edges and let
\[
 \PE:E(\widehat{\mathcal T})\longrightarrow E
\]
be the quotient map.  The valence of $e\in E$ is
\[
 v(e)=\#\PE^{-1}(e).
\]
We regard $\PE^{-1}(e)$ as the set of local occurrences of $e$.

For later use, define the edge star as the occurrence set
\[
 \St(e)=\{(\widehat\sigma,\widehat e)\mid\widehat e\subset\widehat\sigma,
 \ \PE(\widehat e)=e\}.
\]
The set of quotient edges appearing in this star is
\[
 \mathcal E(\St(e))
 =\{\PE(\widehat f)\mid\widehat f\subset\widehat\sigma,
 (\widehat\sigma,\widehat e)\in\St(e)\}.
\]
The occurrence formulation is necessary because, in a pseudo-manifold, the same quotient tetrahedron may occur more than once around an edge.

Now let $N$ be a compact $3$-manifold with boundary components $S_1,\dots,S_k$.  Let $C(N)$ be the quotient obtained by collapsing each $S_i$ to a point $v_i$.  An ideal triangulation of $N$ is a triangulation $\mathcal T_N$ of $C(N)$ whose vertex set is precisely $\{v_1,\dots,v_k\}$.  Then $C(N)\setminus\{v_1,\dots,v_k\}\cong N\setminus\partial N$, and $(C(N),\mathcal T_N)$ is a closed pseudo $3$-manifold in the above sense.

\subsection{Hyper-ideal tetrahedra}\label{subsec:hyperideal-terminology}

A hyper-ideal tetrahedron is a compact convex polyhedron in $\mathbb H^3$ combinatorially equivalent to a truncated tetrahedron. Its faces lie in totally geodesic planes, and its four triangular truncation faces meet the adjacent hexagonal faces orthogonally. Thus its edges are geodesic segments and its four hexagonal faces are right-angled; see \cite[Section~4.1]{LuoYang}.  Label the vertices of its underlying combinatorial tetrahedron by $0,1,2,3$. The internal edge joining the truncation faces at $i$ and $j$ has length $l_{ij}=l_{ji}>0$ and dihedral angle $\alpha_{ij}=\alpha_{ji}\in(0,\pi)$ ($i\ne j$).  Bao--Bonahon's angle characterization states that a vector $(\alpha_{ij})\in\R_{>0}^6$ is the dihedral-angle vector of a hyper-ideal tetrahedron if and only if
\[
 \sum_{j\ne i}\alpha_{ij}<\pi\qquad (i=0,1,2,3).
\]
The isometry class is determined either by the 6 dihedral angles or by the 6 edge lengths; see \cite{BaoBonahon,LuoYang}.  Let
\[
 L\subset\R_{>0}^6
\]
be the open set of positive length vectors realized by nondegenerate hyper-ideal tetrahedra. A \emph{genuine hyper-ideal tetrahedron} is an actual nondegenerate hyperbolic polyhedron of the kind described above, and a \emph{genuine length vector} is an element of $L$. A \emph{generalized hyper-ideal tetrahedron} is specified by an arbitrary $l\in\R_{>0}^6$, which need not belong to $L$. Its extended angles are defined in \Cref{subsec:generalized-angles}; it is genuine precisely when $l\in L$, or equivalently when \eqref{eq:length-criterion} holds.

A global edge length vector $l=(l_e)_{e\in E}\in\R_{>0}^E$ induces the local assignment $\widehat l=l\circ\PE$. We define
\[
 L(M,\mathcal T)=\{l\in\R_{>0}^E\mid\widehat l|_{E(\widehat\sigma)}\in L
 \text{ for every }\widehat\sigma\in\widehat{\mathcal T}\}.
\]
A \emph{genuine hyper-ideal metric} is a global length vector $l\in L(M,\mathcal T)$, so that every local tetrahedron is genuine. When $l\in L(M,\mathcal T)$, corresponding right-angled hexagons have equal alternating edge lengths and therefore are glued by isometries realizing the prescribed vertex correspondence.

For $l\in L(M,\mathcal T)$ and a local edge $\widehat e$ in a tetrahedron $\widehat\sigma$, let $\alpha(\widehat e;l)$ denote its dihedral angle for the induced length vector $\widehat l|_{E(\widehat\sigma)}$. The \emph{combinatorial Ricci curvature} is the cone-angle defect
\[
 K_e(l)=2\pi-\sum_{\widehat e\in\PE^{-1}(e)}\alpha(\widehat e;l),
 \qquad e\in E.
\]
Thus $K:L(M,\mathcal T)\to\R^E$ records the curvature at every quotient edge.

\begin{definition}\label{def:geometric}
A triangulation $(M,\mathcal T)$ is \emph{geometric} if some $l\in L(M,\mathcal T)$ has total dihedral angle $2\pi$ around every quotient edge, counting local occurrences with multiplicity. Equivalently, $K(l)=0$; such an $l$ is a \emph{zero-curvature hyper-ideal metric}. For an ideal triangulation of a compact manifold, this is equivalent to a geometric decomposition of a complete hyperbolic metric, as follows.
\end{definition}

\begin{proposition}\label{prop:geometric-realization}
Let $\mathcal T_N$ be an ideal triangulation of a compact $3$-manifold $N$ with nonempty boundary, and let $l\in L(C(N),\mathcal T_N)$. Glue the corresponding truncated tetrahedra by the prescribed isometries of their hexagonal faces. The resulting metric on $N$ is smooth and hyperbolic with totally geodesic boundary if and only if every quotient-edge angle sum is $2\pi$. When this holds, the metric is complete on $N$, including its boundary.
\end{proposition}
\begin{proof}
Apply the consistency criterion for fully truncated tetrahedra in \cite[Theorem~2.13 and Remark~2.14]{FrigerioPetronio}; see also \cite[Definition~1.2 and the discussion following Definition~1.3]{FengGeHua}. The global length vector ensures that the hexagonal faces match. Orthogonality makes the truncation triangles meet without bending along their sides; their total angle at a boundary vertex is the angle sum at the corresponding quotient edge. The angle-sum equations therefore remove the cone singularities. For nonorientable $N$, lift the prescribed gluing to the orientation double cover and descend the local conclusion. The glued length space is compact, hence complete with its boundary included.
\end{proof}

\subsection{The length formula and generalized angles}\label{subsec:generalized-angles}

Following \cite[Definition~3.1]{Zhao}, choose a compatible ordering of the six local edges of a tetrahedron (called an \emph{edge orientation})
\[
 (e_1,e_2,e_3,e_4,e_5,e_6)
\]
so that $(e_1,e_4)$, $(e_2,e_5)$, and $(e_3,e_6)$ are opposite pairs and $e_1,e_2,e_3$ have a common endpoint. For example, for vertices $0,1,2,3$ one may take
\[
 (e_1,e_2,e_3,e_4,e_5,e_6)=(01,02,03,23,13,12).
\]
Put $x_i=\cosh l_i$. For the dihedral angle at $e_1$, use the length formula of Luo--Yang \cite[Lemma~4.3]{LuoYang}, written in this ordering as in \cite[Eq.~(3.1)]{Zhao}:
\begin{equation}\label{eq:phi}
 \varphi(x_1,\ldots,x_6)=
 \frac{x_2x_3+x_5x_6+x_1x_2x_5+x_1x_3x_6-x_1^2x_4+x_4}
 {\sqrt{2x_1x_2x_6+x_1^2+x_2^2+x_6^2-1}
  \sqrt{2x_1x_3x_5+x_1^2+x_3^2+x_5^2-1}}.
\end{equation}
For each $i$, define $\varphi_i(x)$ by the same formula after relabelling the edges so that $e_i$ is the target edge; thus $\varphi_1=\varphi$. The value is independent of the compatible ordering of the other edges. For a length vector $l\in\R_{>0}^6$, write $\varphi_i(l)=\varphi_i(\cosh l_1,\ldots,\cosh l_6)$. A genuine hyper-ideal tetrahedron has $\alpha_i=\arccos\varphi_i(l)$. Luo--Yang's length criterion \cite[Proposition~4.4]{LuoYang} is
\begin{equation}\label{eq:length-criterion}
 l\in L\quad\Longleftrightarrow\quad
 \varphi_i(l)\in(-1,1)\quad(i=1,\ldots,6).
\end{equation}

For $l\in\R_{\ge0}^6$, use the convention of \cite[Definition~2.3]{Zhao}: set
\[
 \clamp_{[-1,1]}(y)=\max\{-1,\min\{y,1\}\},
\]
and define the \emph{extended dihedral angles} by
\begin{equation*}
 \alpha_i(l)=\arccos\bigl(\clamp_{[-1,1]}(\varphi_i(l))\bigr)\in[0,\pi].
\end{equation*}
The maps $\alpha_i:\R_{\ge0}^6\to[0,\pi]$ are continuous by \cite[Corollary~4.9; see also Lemma~4.7]{LuoYang}. In the explicit formula this also follows because the two radicands are positive on $[1,\infty)^6$.  We repeatedly use the elementary implication
\begin{equation}\label{eq:angle-lower-principle}
 \varphi_i\le\Phi
 \quad\Longrightarrow\quad
 \alpha_i\ge\arccos\bigl(\clamp_{[-1,1]}\Phi\bigr),
\end{equation}
because $\clamp_{[-1,1]}$ is nondecreasing and $\arccos$ is decreasing.

\subsection{Curvature, co-volume, and the flow}

For $l\in\R_{>0}^E$, use the same notation $\alpha(\widehat e;l)$ for the extended angle in the unique tetrahedron containing $\widehat e$, evaluated at its induced six-tuple. Define the generalized curvature map $\Ktil:\R_{>0}^E\to\R^E$ by
\begin{equation*}
 \Ktil_e(l)=2\pi-\sum_{\widehat e\in\PE^{-1}(e)}\alpha(\widehat e;l),\qquad e\in E.
\end{equation*}
If $l\in L(M,\mathcal T)$, then $\Ktil_e(l)=K_e(l)$. For a cosh-length assignment $x$, write $\alpha(\widehat e;x)$ for the angle at the corresponding lengths $l_f=\arccosh x_f$. We omit the assignment from $\alpha(\widehat e)$ only when it is fixed by the context.

Luo introduced a combinatorial Ricci flow for hyper-ideal metrics \cite{Luo2005}. We use the extended flow of Feng--Ge--Hua \cite[Eq.~(1.3)]{FengGeHua}, given by
\begin{equation}\label{eq:flow}
 \frac{\dd}{\dd t}l_e(t)=\Ktil_e(l(t))l_e(t)
 \qquad\text{for every }e\in E,\quad t\ge0.
\end{equation}
It has a unique solution for all $t\ge0$ from any $l(0)\in\R_{>0}^E$ \cite[Theorem 1.5]{FengGeHua}.  In the variables $x_e=\cosh l_e$,
\begin{equation}\label{eq:xflow}
 \dot x_e(t)=\sinh(l_e(t))l_e(t)\Ktil_e(l(t)).
\end{equation}
Thus, if $x_e$ lies on an upper boundary and the incident angle sum is larger than $2\pi$, then $\Ktil_e<0$ and $\dot x_e<0$, since $\sinh(l_e)l_e>0$. The flow therefore points inward in that coordinate.

For a local tetrahedron $\widehat\sigma$, let $\cov_{\widehat\sigma}:\R^6\to\R$ be the extended co-volume of Luo--Yang \cite[Corollary~4.12]{LuoYang}. On genuine length vectors $u\in L$, it agrees with
\[
 \cov_{\widehat\sigma}(u)=2\Vol(u)+\sum_{i=1}^6\alpha_i(u)u_i.
\]
The extension uses $\alpha_i(u^+)$ for $u\in\R^6$, where $u_i^+=\max\{u_i,0\}$, and is $C^1$ and convex. Define
\begin{align*}
 \cov:\R^E&\to\R,\qquad
 \cov(l)=\sum_{\widehat\sigma\in\widehat{\mathcal T}}
 \cov_{\widehat\sigma}\bigl((l\circ\PE)|_{E(\widehat\sigma)}\bigr),\\
 H:\R^E&\to\R,\qquad H(l)=\cov(l)-2\pi\sum_{e\in E}l_e.
\end{align*}
The function $H$ is also $C^1$ and convex. For $l\in\R_{>0}^E$, differentiation gives \cite[Eq.~(4.1)]{FengGeHua}
\begin{equation*}
 \frac{\partial H}{\partial l_e}(l)=-\Ktil_e(l),\qquad e\in E.
\end{equation*}
Consequently,
\begin{equation}\label{eq:Hdecrease}
 \frac{\dd}{\dd t}H(l(t))
 =-\sum_{e\in E}\Ktil_e(l(t))^2l_e(t)\le0.
\end{equation}

We use the following edge-dependent compactness principle; its proof adapts \cite[Proposition~2.13]{Zhao} and records that the zero-curvature vector lies in the same box.

\begin{proposition}\label{prop:compactness}
Let $l:[0,\infty)\to\R_{>0}^E$ be a solution of \eqref{eq:flow}. Suppose there are constants $c_e,C_e$ ($e\in E$), with $0<c_e\le C_e<\infty$, such that
\[
 c_e\le l_e(t)\le C_e\qquad\text{for every }e\in E\text{ and }t\ge0.
\]
Then there exists $l^*\in D:=\prod_{e\in E}[c_e,C_e]$ such that $\Ktil(l^*)=0$.
\end{proposition}
\begin{proof}
The set $D$ is compact. The continuous function $H$ is bounded on $D$, and \eqref{eq:Hdecrease} shows that $H(l(t))$ is nonincreasing, hence converges to a finite limit. For each integer $n\ge0$, the mean value theorem gives $t_n\in(n,n+1)$ with
\[
 \frac{\dd}{\dd t}H(l(t))\bigg|_{t=t_n}
 =H(l(n+1))-H(l(n))\longrightarrow0\qquad(n\to\infty).
\]
Therefore \eqref{eq:Hdecrease} gives
\[
 \sum_{f\in E}\Ktil_f(l(t_n))^2l_f(t_n)\longrightarrow0\qquad(n\to\infty).
\]
Each summand is nonnegative, and $l_e(t_n)\ge c_e>0$, so $\Ktil_e(l(t_n))\to0$ as $n\to\infty$ for every $e\in E$.
A subsequence of $l(t_n)$ converges to some $l^*\in D$. Continuity of the extended angles, hence of $\Ktil$, gives $\Ktil(l^*)=0$.
\end{proof}

\subsection{Conventions for explicit face pairings}\label{subsec:encoding}
Every abstract tetrahedron in an explicit construction has vertices $0,1,2,3$. Its face $F_f$ is opposite vertex $f$. A record $(u,p)$ in column $f$ of tetrahedron $t$ specifies
\[
 F_f^{T_t}\longrightarrow F_{p(f)}^{T_u},\qquad j\longmapsto p(j)\quad(j\ne f).
\]
The string $p(0)p(1)p(2)p(3)$ records a permutation, not a cyclic notation. The omitted vertex $f$ is not glued by this face map: its image $p(f)$ records only the target face number. The reverse record must use $p^{-1}$. Self-pairings of distinct faces of one tetrahedron are permitted, but a face is never paired with itself. Quotient edges are equivalence classes of local edges under these face maps. The complete finite records used below, together with the link-checking procedure, are in Appendix~\ref{app:examples}.

\section{Invariant boxes and local bichromaticity}\label{sec:compactness}

Using the compactness principle from \Cref{prop:compactness}, we prove an invariant-box theorem and derive the local bichromatic criterion. The final subsection gives its cyclic examples.

\subsection{The invariant-box theorem}\label{subsec:abstract-box}

We now formulate an abstract box criterion in the $x$-coordinates.

\begin{definition}\label{def:theta-box}
For constants $1<L_e<U_e\le3$, set
\[
 Q=\prod_{e\in E}[L_e,U_e].
\]
For an edge $e$, define
\[
 \Theta_e^+(Q)=\inf_{\substack{x\in Q\\x_e=U_e}}
 \sum_{\widehat e\in\PE^{-1}(e)}\alpha(\widehat e;x),
 \qquad
 \Theta_e^-(Q)=\sup_{\substack{x\in Q\\x_e=L_e}}
 \sum_{\widehat e\in\PE^{-1}(e)}\alpha(\widehat e;x).
\]
\end{definition}

The following criterion abstracts the barrier and compactness arguments of Feng--Ge--Hua \cite[Theorems~5.1--5.3 and proof of Theorem~1.9]{FengGeHua}; compare also Zhao \cite[Theorem~5.1 and proof of Theorem~1.5]{Zhao}. Here the inward-pointing conditions are stated directly as inequalities for the angle sums on the faces of a prescribed box.

\begin{theorem}\label{thm:abstract-box}
Assume that, for every quotient edge $e\in E$,
\[
 \Theta_e^+(Q)>2\pi,
 \qquad
 \Theta_e^-(Q)<2\pi.
\]
Then every flow with $x(0)\in\operatorname{int}Q$ satisfies $x(t)\in Q$ for all $t\ge0$.  Consequently, $(M,\mathcal T)$ admits a unique zero-curvature hyper-ideal metric, and the extended combinatorial Ricci flow converges to it exponentially from any positive initial condition.
\end{theorem}

\begin{proof}
At a first contact with an upper face $x_e=U_e$, the angle sum is larger than $2\pi$, so $\Ktil_e<0$ and \eqref{eq:xflow} gives $\dot x_e<0$, contradicting $\dot x_e\ge0$ at first contact from below.  At a first contact with a lower face $x_e=L_e$, the angle sum is smaller than $2\pi$, so $\dot x_e>0$, contradicting $\dot x_e\le0$. Hence the trajectory cannot reach the boundary at a finite first contact time, and $x(t)\in Q$ for all $t\ge0$.

By \Cref{prop:compactness}, a generalized zero-curvature metric $l^*$ exists with $x(l^*)\in Q$. Thus $1<x_e(l^*)\le3$, or equivalently $0<l_e^*\le\arccosh3$, for every $e\in E$.  Feng--Ge--Hua's small-length criterion \cite[Theorem~3.9]{FengGeHua} implies that every local tetrahedron is genuine hyper-ideal; equivalently, \eqref{eq:length-criterion} holds.  Hence $l^*\in L(M,\mathcal T)$.  Uniqueness follows from rigidity \cite[Theorem 1.2]{LuoYang}. Once a genuine zero-curvature metric exists, convergence from any positive initial vector, not only those in $Q$, follows from \cite[Theorem 1.6]{FengGeHua}.
\end{proof}

\subsection{A local bichromatic theorem}\label{subsec:bichromatic}

To deduce \Cref{thm:local-bichromatic} from \Cref{thm:abstract-box}, we first define local bichromaticity and prove the tetrahedral angle-sum estimate in \Cref{lem:bichromatic-tetrahedron}. This theorem permits valence 7 and uses a common upper cosh-length bound of 3, without the valence-dependent boxes of the later sections.

\begin{definition}\label{def:local-bichromatic}
A tetrahedron $\widehat\sigma\in\widehat{\mathcal T}$ is \emph{bichromatic} if
\[
 \#\PE\bigl(E(\widehat\sigma)\bigr)\le2.
\]
Thus its 6 local edges represent at most 2 quotient-edge classes. We say that $(M,\mathcal T)$ is \emph{locally bichromatic at low-valence edges} if every tetrahedron containing an occurrence of an edge $e\in E$ with $7\le v(e)\le9$ is bichromatic.
\end{definition}

The local estimate needed below is a sum inequality inside a single tetrahedron.  Label the vertices of a combinatorial tetrahedron by $1,2,3,4$ and identify its 6 edges with the edges of $K_4$.  For a nonempty subgraph $G\subset K_4$, call the edges of $G$ the $a$-edges and the complementary edges the $b$-edges.  Assign cosh-length $a$ to the former and cosh-length $b$ to the latter. Figure~\ref{fig:bichromatic-patterns} shows the 10 nonempty subgraphs up to isomorphism.

\begin{figure}[!htbp]
 \centering
 \includegraphics[width=\textwidth]{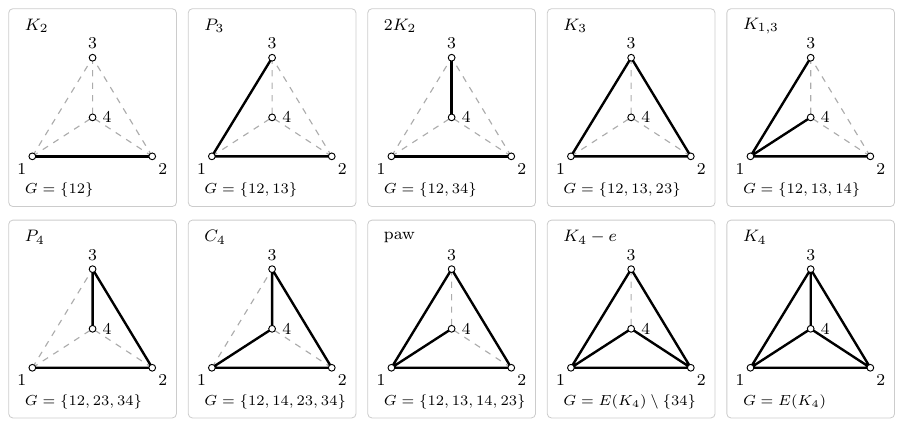}
 \caption{The 10 isomorphism classes of nonempty subgraphs $G\subset K_4$.  Solid black edges are the $a$-edges, namely the local occurrences of the distinguished quotient edge; dashed gray edges are the $b$-edges.}
 \label{fig:bichromatic-patterns}
\end{figure}

\begin{lemma}\label{lem:bichromatic-tetrahedron}
Let $G\subset K_4$ be nonempty and let $1<b\le a\le3$.  Give the edges in $G$ cosh-length $a$ and the remaining edges cosh-length $b$.  If $\alpha_{ij}(a,b)$ denotes the dihedral angle at $ij$, then
\begin{equation}\label{eq:bichromatic-sum}
 \sum_{ij\in E(G)}\alpha_{ij}(a,b)
 \ge |E(G)|\,\beta(a),
 \qquad
 \beta(a):=\arccos\frac{a}{2a-1}.
\end{equation}
Equality holds when $b=a$.
\end{lemma}

\begin{proof}
Since all cosh-lengths are at most $3$ and are strictly larger than one, the corresponding generalized tetrahedron is genuine hyper-ideal by the small-length criterion of Feng--Ge--Hua \cite[Theorem~3.9]{FengGeHua}.  Hence its dihedral angles are ordinary smooth functions of $a$ and $b$.

Set
\[
 A_G(a,b)=\sum_{ij\in E(G)}\alpha_{ij}(a,b).
\]
At $b=a$ the tetrahedron is regular.  Substitution in \eqref{eq:phi} gives
\[
 \cos\alpha_{ij}
 =\frac{a(a+1)^2}{2a^3+3a^2-1}
 =\frac{a}{2a-1},
\]
so that
\begin{equation}\label{eq:regular-target-sum}
 A_G(a,a)=|E(G)|\beta(a).
\end{equation}
It therefore suffices to prove
\begin{equation}\label{eq:bichromatic-monotonicity}
 \frac{\partial A_G}{\partial b}(a,b)\le0
 \qquad(1<b\le a\le3).
\end{equation}

We verify the estimate for all 10 isomorphism classes, since the claim concerns the sum of all target angles. For the 3 non-edge-transitive graphs, a single target angle need not be monotone even though their weighted sum is.

Put
\[
 X=a+2b^2-1,
 \qquad
 Y=2a^2+b-1.
\]
For the 7 edge-transitive graphs, every target edge has a common cosine $p_G(a,b)$.  Direct substitution in \eqref{eq:phi} and differentiation give Table~\ref{tab:bichromatic-edge-transitive}.
\begin{table}[t]
\centering
\normalsize
\renewcommand{\arraystretch}{1.25}

\begin{tabular}{c|c|c}
\toprule
$G$ & $p_G(a,b)$ & $\partial p_G/\partial b$\\
\midrule
$K_2$
& $\displaystyle\frac{b(2b+1-a)}{X}$
& $\displaystyle\frac{(a-1)(2b^2+4b+1-a)}{X^2}$\\[2mm]
$P_3$
& $\displaystyle\frac{b\sqrt{(a+1)(b+1)}}{\sqrt X\sqrt Y}$
& $\displaystyle\frac{(a-1)\sqrt{a+1}\,Q(a,b)}{\sqrt{b+1}\,X^{3/2}Y^{3/2}}$\\[2mm]
$2K_2$
& $\displaystyle\frac{2b^2+a-a^2}{X}$
& $\displaystyle\frac{4b(a^2-1)}{X^2}$\\[2mm]
$K_3$
& $\displaystyle\frac{b\sqrt{a+1}}{\sqrt{2a-1}\sqrt X}$
& $\displaystyle\frac{(a-1)\sqrt{a+1}}{\sqrt{2a-1}\,X^{3/2}}$\\[2mm]
$K_{1,3}$
& $\displaystyle\frac{a^2+b}{Y}$
& $\displaystyle\frac{a^2-1}{Y^2}$\\[2mm]
$C_4$
& $\displaystyle\frac{a(b+1)}{Y}$
& $\displaystyle\frac{2a(a^2-1)}{Y^2}$\\[2mm]
$K_4$
& $\displaystyle\frac{a}{2a-1}$
& $0$\\
\bottomrule
\end{tabular}
\caption{Cosines and derivatives for the 7 edge-transitive target graphs.}
\label{tab:bichromatic-edge-transitive}
\end{table}
Here
\[
 Q(a,b)=3a^2b+2a^2+2ab^3+2b^3+b^2-b-1>0.
\]
All entries in the third column are nonnegative on $1<b\le a\le3$.  Since $\arccos$ is decreasing, \eqref{eq:bichromatic-monotonicity} follows for these 7 graphs.

It remains to treat $P_4$, the paw graph, and $K_4-e$.  The following identities are obtained by substituting the corresponding two-valued length vectors in \eqref{eq:phi}, differentiating the sum of the target angles, and clearing the positive square-root factors.

For $G=P_4$, define
\begin{align*}
 R&=-a^2+3ab+a-b^2+b+1,\\
 P&=a^2+ab+a+b^2+b-1,
\end{align*}
and
\begin{align*}
 S={}&8a^3b+2a^3-4a^2b^2+8a^2b+4a^2
      +8ab^3+8ab^2-5ab-3a\\
    &\hspace{34mm}+2b^3+4b^2-3b-1.
\end{align*}
The two end edges have a common angle by symmetry. Thus $A_{P_4}$ is twice this common angle plus the middle-edge angle, and
\begin{equation*}
 -\frac{\partial A_{P_4}}{\partial b}
 =\frac{\sqrt{a^2-1}\,S}{XY\sqrt{RP}}.
\end{equation*}
The polynomial $P=a^2+ab+a+b^2+b-1$ is positive directly from $a,b>1$. To check $R$ and $S$, put $r=a-b$ and $s=b-1$.  Then $r,s\ge0$ and $r+s\le2$, while
\[
 R=r(2+s-r)+(s+2)^2>0,
\]
and
\begin{align*}
 S={}&8r^3s+10r^3+20r^2s^2+54r^2s+38r^2
      +24rs^3+102rs^2\\
    &\quad+135rs+54r+12s^4+68s^3+135s^2+108s+28>0.
\end{align*}
Thus $\partial_bA_{P_4}<0$.

For the paw graph, put
\[
 J=-a^3+4a^2b+3a^2+a+2b^2-1.
\]
The three edge orbits have sizes $2,1,1$, and the angle is constant on each orbit. Thus $A_{\mathrm{paw}}$ is the sum of the three representative angles with these weights, and
\begin{equation*}
 -\frac{\partial A_{\mathrm{paw}}}{\partial b}
 =\frac{2\sqrt{a-1}}{\sqrt J}
 \left(\frac{a+1}{Y}+\frac{a(b+1)}{X}\right)>0.
\end{equation*}
Indeed, $-a^3+4a^2b=a^2(4b-a)>0$ because $b>1$ and $a\le3$, and hence $J>0$.

Finally, for $G=K_4-e$, put
\[
 D=4a^2-ab+a+b-1.
\]
The target edge opposite the missing edge forms a singleton orbit, and the other four target edges have a common angle. Thus $A_{K_4-e}$ is the former angle plus four times the latter, and
\begin{equation*}
 -\frac{\partial A_{K_4-e}}{\partial b}
 =\frac{\sqrt{a-1}}{\sqrt{b+1}\sqrt D}
 \left(\frac{4a(a+1)}{Y}-1\right)>0.
\end{equation*}
Here
\[
 D\ge3a^2+2a-1>0,
 \qquad
 4a(a+1)-Y=2a^2+4a-b+1>0.
\]
This proves \eqref{eq:bichromatic-monotonicity} in all 10 cases.  Integrating from $b$ to $a$ and using \eqref{eq:regular-target-sum} proves \eqref{eq:bichromatic-sum}.
\end{proof}

We shall use the endpoint $a=3$.  The resulting regular angle satisfies
\begin{equation}\label{eq:beta-three}
 \beta(3)=\arccos\frac35>\frac{2\pi}{7}.
\end{equation}
For completeness, if $c=\cos(2\pi/7)$, then $8c^3+4c^2-4c-1=0$.  The polynomial $8t^3+4t^2-4t-1$ is strictly increasing for $t\ge1/2$ and takes the value $-29/125$ at $t=3/5$.  Since $c>1/2$, this gives $3/5<c$, hence \eqref{eq:beta-three}.

\begin{theorem}\label{thm:local-bichromatic}\label{thm:intro-bichromatic}
Let $(M,\mathcal T)$ be a closed pseudo $3$-manifold.  Assume that $v(e)\ge7$ for every quotient edge $e\in E$
and that $(M,\mathcal T)$ is locally bichromatic at every edge of valence 7, 8, or 9 in the sense of \Cref{def:local-bichromatic}.  Then $(M,\mathcal T)$ admits a unique zero-curvature hyper-ideal metric.  The extended combinatorial Ricci flow converges exponentially to this metric from any positive initial condition.
\end{theorem}

\begin{proof}
Let
\[
 V_{\max}=\max_{e\in E}v(e).
\]
We construct an invariant box with common upper endpoint $3$.  First observe that, for the angle at the first local edge,
\[
 \varphi(1,x_2,x_3,x_4,x_5,x_6)=1
 \qquad(1\le x_2,\ldots,x_6\le3),
\]
because both the numerator and denominator of \eqref{eq:phi} reduce to
\[
 (x_2+x_6)(x_3+x_5).
\]
Consequently the corresponding extended angle is zero.  Define
\[
 \eta(L)=
 \sup_{\substack{1\le x_1\le L\\1\le x_2,\ldots,x_6\le3}}
 \alpha_1(x_1,\ldots,x_6).
\]
The extended-angle function is continuous, hence uniformly continuous, on the compact cube $[1,3]^6$.  Since $\alpha_1(1,x_2,\ldots,x_6)=0$ uniformly in the other coordinates, this uniform continuity implies
\[
 \eta(L)\longrightarrow0\qquad(L\downarrow1).
\]
Choose $L>1$ so close to one that
\begin{equation*}
 V_{\max}\eta(L)<2\pi,
\end{equation*}
and set
\[
 Q_L=[L,3]^E.
\]
On a lower face $x_e=L$, every local angle at an occurrence of $e$ is at most $\eta(L)$.  Thus
\[
 \sum_{\widehat e\in\PE^{-1}(e)}\alpha(\widehat e)
 \le v(e)\eta(L)
 \le V_{\max}\eta(L)<2\pi.
\]
Hence $\Theta_e^-(Q_L)<2\pi$ for every edge.

Now let $x_e=3$.  Suppose first that $7\le v(e)\le9$.  For each tetrahedron $\widehat\sigma$ containing one or more occurrences of $e$, let $G_{\widehat\sigma}\subset K_4$ consist of the local edges projecting to $e$.  By local bichromaticity, all complementary local edges, if present, project to one quotient edge and therefore have one common cosh-length $b\in[L,3]$.  Applying \Cref{lem:bichromatic-tetrahedron} in each tetrahedron and summing gives
\begin{align*}
 \sum_{\widehat e\in\PE^{-1}(e)}\alpha(\widehat e)
 &\ge
 \sum_{\widehat\sigma}|E(G_{\widehat\sigma})|\,\beta(3)\\
 &=v(e)\beta(3)
 \ge7\arccos\frac35
 >2\pi.
\end{align*}

Suppose instead that $v(e)\ge10$.  On the upper face $x_e=3$, every local occurrence of $e$ is a longest edge of its tetrahedron.  Feng--Ge--Hua's longest-edge estimate gives
\[
 \alpha(\widehat e)>\frac\pi5
\]
for every such occurrence \cite[Corollary~3.7]{FengGeHua}.  Therefore
\[
 \sum_{\widehat e\in\PE^{-1}(e)}\alpha(\widehat e)
 >v(e)\frac\pi5\ge2\pi.
\]
Each upper face of $Q_L$ is compact, and its angle sum is continuous. Thus the pointwise strict inequality also gives $\Theta_e^+(Q_L)>2\pi$ for every edge.  The conclusion now follows from the invariant-box criterion, \Cref{thm:abstract-box}.
\end{proof}

\subsection{Examples: a locally bichromatic cyclic family}\label{subsec:bichromatic-family}

The following cyclic family, with face pairings specified in the proof, is geometric by \Cref{thm:local-bichromatic}. It uses repeated low-valence occurrences inside a tetrahedron, so the symmetric valence tables of \Cref{sec:C2} do not detect it.

\begin{proposition}\label{prop:bichromatic-family}
For every integer $d\ge1$ there is a connected orientable one-vertex ideal triangulation $\mathcal S_d$ with $3d$ tetrahedra and quotient edges $a_0,\ldots,a_{d-1},b$, with
\[
 v(a_k)=7\quad(k\in\Z/d\Z),\qquad v(b)=11d.
\]
Its unique vertex link has genus $2d$. For each $k\in\Z/d\Z$, every tetrahedron $A_k,B_k,C_k$ uses only the quotient-edge classes $a_k$ and $b$, so $\mathcal S_d$ is geometric by \Cref{thm:local-bichromatic}.
\end{proposition}

\begin{proof}
For each $k\in\Z/d\Z$ take tetrahedra $A_k,B_k,C_k$, each with vertices $0,1,2,3$. Figure~\ref{fig:bichromatic-coupling} shows the cyclic coupling. Use the face convention of \Cref{subsec:encoding} and the following 6 pairings, including their inverses:
\begin{center}
\begin{tabular}{ccc}
\toprule
Source face&Target face&Permutation\\
\midrule
$A_k:F_0$&$C_k:F_2$&\code{2031}\\
$A_k:F_1$&$B_k:F_2$&\code{3201}\\
$A_k:F_2$&$B_k:F_3$&\code{2031}\\
$A_k:F_3$&$C_k:F_3$&\code{1023}\\
$B_k:F_0$&$C_{k+1}:F_1$&\code{1302}\\
$B_k:F_1$&$C_k:F_0$&\code{3012}\\
\bottomrule
\end{tabular}
\end{center}
Every local face occurs in exactly one pair. The dual graph, with a vertex for each tetrahedron and an edge for each paired face, is connected. The local edges assigned to the two types are
\begin{center}
\begin{tabular}{ccc}
\toprule
Tetrahedron&$a_k$-type edges&$b$-type edges\\
\midrule
$A_k$&$01,13,23$&$02,03,12$\\
$B_k$&$01,02$&$03,12,13,23$\\
$C_k$&$01,13$&$02,03,12,23$\\
\bottomrule
\end{tabular}
\end{center}
Each face map preserves these types. The $a_k$-edges form the cycle
\[
 A_k^{01}\sim B_k^{02}\sim C_k^{13}\sim A_k^{23}
 \sim B_k^{01}\sim A_k^{13}\sim C_k^{01}\sim A_k^{01}.
\]
The only pairing changing $k$ contains no $a$-type edge. Thus each displayed cycle is one quotient-edge class $a_k$ of valence 7, and the $d$ classes are distinct.

For the other type, the face maps give
\begin{align*}
 A_k^{12}&\sim C_k^{03}\sim B_{k-1}^{12}\sim A_{k-1}^{03}
 \sim B_{k-1}^{13}\sim C_k^{23}\\
 &\sim B_k^{03}\sim A_k^{02}\sim C_k^{12}\sim B_k^{23}
 \sim C_{k+1}^{02}\sim A_{k+1}^{12}.
\end{align*}
Thus all $A_k^{12}$ lie in one class, since $k\mapsto k+1$ has one orbit on $\Z/d\Z$. As $k$ varies, these chains contain every $b$-type local edge: those with shifted subscripts occur in the neighbouring chain. There are $3+4+4=11$ per copy, so the single class $b$ has valence $11d$.

All 6 permutations are odd, so the orders $[0,1,2,3]$ orient the tetrahedra compatibly. The pairings within one copy already identify its 12 local vertices, and the cross-copy pairing identifies these vertices for consecutive $k$. Thus the quotient has one vertex. There is no reversed edge identification: compatible local orientations are
\begin{align*}
 a_k:
 &\quad A_k^{10},A_k^{13},A_k^{23},B_k^{01},B_k^{02},C_k^{01},C_k^{31},\\
 b:
 &\quad A_k^{02},A_k^{03},A_k^{12},B_k^{30},B_k^{21},B_k^{31},B_k^{23},
 C_k^{02},C_k^{03},C_k^{12},C_k^{23}.
\end{align*}
Each face map preserves these directed edges, so no edge is identified with itself in reverse. The cycles above show that every edge link is a circle. Let $S$ be the vertex link obtained by gluing the truncation triangles. Each vertex of $S$ corresponds to one end of an oriented quotient edge. Its corner-link intervals are paired by the same face maps as the transverse intervals around that edge, so its link is the corresponding edge-link circle. Hence $S$ is a closed surface. The face identifications connecting all local vertices induce side-adjacency chains connecting all truncation triangles, so $S$ is connected. The compatible tetrahedral orientations induce an orientation on $S$. Its cell counts are $2(d+1),18d,12d$, so its Euler characteristic is $2-4d$, and its genus is $2d$. The hypotheses of \Cref{thm:local-bichromatic} now hold.
\end{proof}

\begin{figure}[!htbp]
\centering
\includegraphics[width=.95\textwidth]{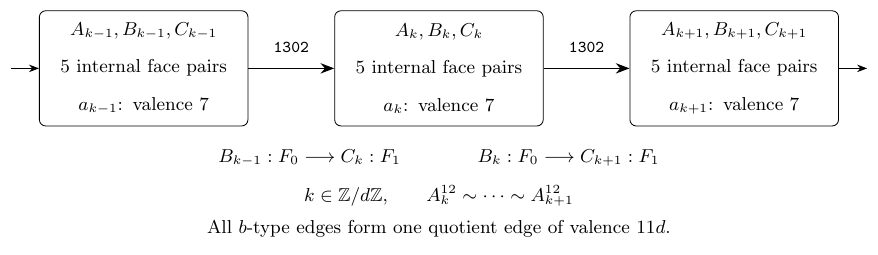}
\caption{The cyclic coupling for $\mathcal S_d$. Each box contains $A_k,B_k,C_k$. Only the displayed face pairing changes the copy index. The low-valence class $a_k$ stays in its box, whereas the chain through $b$-type edges connects all boxes.}
\label{fig:bichromatic-coupling}
\end{figure}

\begin{corollary}\label{cor:bichromatic-family}
The compact manifolds associated with $\mathcal S_d$ are pairwise nonhomeomorphic. Their triangulations are geometric, but satisfy neither of the symmetric $C=2$ valence criteria in Tables~\ref{tab:C2-simple} and~\ref{tab:C2-refined}.
\end{corollary}
\begin{proof}
Their connected boundary genera $2d$ differ. In $A_k$, the target $01$ and its adjacent edge $13$ both represent $a_k$, and the opposite edge $23$ also represents $a_k$. Hence the target valence is 7, the opposite valence is 7, and an adjacent valence is 7. The required adjacent thresholds in Tables~\ref{tab:C2-simple} and~\ref{tab:C2-refined} are 14 and 12, respectively. Thus neither table applies.
\end{proof}

\section{The classical value \texorpdfstring{$C=2$}{C=2}}\label{sec:C2}

This section uses Zhao's $C=2$ conditional estimates to obtain explicit conditions on neighbouring valences and construct triangulations satisfying them.

\subsection{Conditional bounds and the edge-star theorem}\label{subsec:C2-theorem}

The conditional estimates in \Cref{lem:C2-conditional} are the preparatory bounds for \Cref{thm:C2-star}.

Zhao's monotonicity lemma \cite[Lemma~3.4]{Zhao} gives, on $[1,2]^6$,
\begin{equation}\label{eq:Zhao-monotonicity}
 \frac{\partial\varphi}{\partial x_j}\ge0
 \quad(j=2,3,5,6),
 \qquad
 \frac{\partial\varphi}{\partial x_4}\le0.
\end{equation}
Set
\begin{equation*}
 \gamma_n=\frac{6}{2+\cos(2\pi/n)}-1\quad(n\ge6),
 \qquad
 b_n=\frac{16}{1+\cos(2\pi/n)}-7\quad(n\ge10),
\end{equation*}
and put $b_n=2$ for $6\le n\le9$. The displayed formulas are from \cite[Eqs.~(3.3)--(3.4)]{Zhao}. The defining endpoint equations are
\[
 \varphi(\gamma_n,1,1,2,1,1)=\cos\frac{2\pi}{n}\quad(n\ge6),
 \qquad
 \varphi(b_n,2,2,1,2,2)=\cos\frac{2\pi}{n}\quad(n\ge10).
\]
The initial intervals are nonempty: $1<\gamma_n<b_n\le2$ for $n\ge6$. Indeed, $c_n:=\cos(2\pi/n)\in[1/2,1)$ gives $1<\gamma_n\le7/5<2$. For $n\ge10$, $c_n\ge c_{10}=(1+\sqrt5)/4>7/9$ implies $b_n<2$, while $b_n-\gamma_n=2(1-c_n)(7+3c_n)/((1+c_n)(2+c_n))>0$.

\begin{lemma}\label{lem:C2-conditional}
Assume $v(f)\ge6$ for every quotient edge $f\in E$. Let $l(t)$ solve \eqref{eq:flow}, put $x_f(t)=\cosh l_f(t)$, and fix $t_0>0$. Suppose
\[
 x_f(t)\le2\quad\text{for every }f\in E\text{ and every }t\in[0,t_0],
\]
and
\[
 \gamma_{v(f)}<x_f(0)<b_{v(f)}\quad\text{for every }f\in E.
\]
Then, for every $t\in[0,t_0]$,
\[
 x_f(t)\ge\gamma_{v(f)}\quad\text{for every }f\in E,
 \qquad
 x_f(t)\le b_{v(f)}<2\quad\text{for every }f\in E\text{ with }v(f)\ge10.
\]
\end{lemma}

\begin{proof}
The lower bound is \cite[Proposition~3.6]{Zhao}. For the upper bound, the proof of \cite[Proposition~3.5]{Zhao} applies edge by edge: at a first upward crossing of a level $b_n<a<2$ for a target of valence $n\ge10$, \eqref{eq:Zhao-monotonicity} gives every incident cosine at most
\[
 \varphi(a,2,2,1,2,2)<\varphi(b_n,2,2,1,2,2)=\cos(2\pi/n).
\]
Thus $\Ktil_e<0$, contradicting upward contact. This argument uses no valence restriction on the other edges. For valences 6--9 the upper bound 2 is already assumed.
\end{proof}

For a local occurrence $\widehat e=e_1$ with opposite edge $e_4$, define
\begin{equation}\label{eq:C2-occurrence}
 \underline\alpha_2(\widehat e)=
 \arccos\!\left(\clamp_{[-1,1]}
 \varphi\bigl(2,b_{v(e_2)},b_{v(e_3)},\gamma_{v(e_4)},
 b_{v(e_5)},b_{v(e_6)}\bigr)\right).
\end{equation}

\begin{theorem}\label{thm:C2-star}
Assume $v(e)\ge6$ for every quotient edge $e\in E$. If, for every $e\in E$ with $v(e)\in\{6,7,8,9\}$,
\[
 \sum_{\widehat e\in\PE^{-1}(e)}\underline\alpha_2(\widehat e)>2\pi,
\]
then $(M,\mathcal T)$ is geometric.  The zero-curvature metric is unique and the extended combinatorial Ricci flow converges to it exponentially from any initial vector $l(0)\in\R_{>0}^E$.
\end{theorem}

\begin{proof}
Following the first-contact argument in \cite[proof of Theorem~5.1]{Zhao}, choose $x(0)\in\prod_{e\in E}(\gamma_{v(e)},b_{v(e)})$ and set $l_e(0)=\arccosh x_e(0)$ for every $e\in E$.  Suppose that some edge first reaches $x=2$ at time $T$.  On $[0,T]$, \Cref{lem:C2-conditional} applies.  An edge of valence at least 10 cannot be the first edge to reach $2$ because it satisfies $x_e\le b_{v(e)}<2$.  For a first-contact edge of valence at most 9, \eqref{eq:Zhao-monotonicity} and the conditional bounds imply that every actual occurrence angle is at least \eqref{eq:C2-occurrence}; hence its total angle is larger than $2\pi$.  Thus $\Ktil_e<0$ and $\dot x_e(T)<0$, a contradiction.

Therefore $x_e(t)<2$ for every $e\in E$ and every $t\ge0$.  Applying the lower part of \Cref{lem:C2-conditional} on every finite interval gives a uniform positive lower bound.  The flow remains in a compact box, so \Cref{prop:compactness} and the $\arccosh3$ criterion complete the proof.
\end{proof}

\subsection{Valence-only corollaries}\label{subsec:C2-corollaries}

The next estimate follows from \eqref{eq:Zhao-monotonicity}. It replaces individual adjacent-edge bounds by a common valence threshold and yields \Cref{cor:C2-explicit,cor:C2-v8}.

For $M\ge6$ and $m\in\{6,7,\ldots\}\cup\{\infty\}$, put $\gamma_\infty=1$ and define
\begin{equation}\label{eq:A2Mm}
 A_2(M,m)=\arccos\!\left(\frac{2b_M^2-\gamma_m}{2b_M^2+1}\right).
\end{equation}
Here $m=\infty$ means that the opposite-edge valence is not used, so only $x_4\ge1$ is retained. The argument of $\arccos$ lies in $(-1,1)$: $b_M\ge1$, $1\le\gamma_m\le7/5$, and the denominator is positive. Notice also that
\[
 A_2(6,m)=A_2(7,m)=A_2(8,m)=A_2(9,m),
\]
since $b_6=b_7=b_8=b_9=2$.

\begin{lemma}\label{lem:C2-symmetric}
Let $n_i\ge6$ be the valence of the quotient edge represented by the local edge $e_i$. Suppose $x_1=2$ and $\gamma_{n_i}\le x_i\le b_{n_i}$ for $i=2,\ldots,6$. If $n_4=m$ and $n_i\ge M$ for $i=2,3,5,6$, then the target angle satisfies
\[
 \alpha_1\ge A_2(M,m).
\]
The same conclusion with $A_2(M,\infty)$ holds without opposite-edge valence information.
\end{lemma}

\begin{proof}
Since $b_n$ is nonincreasing in $n$, the four adjacent cosh-lengths are at most $b_M$, while $x_4\ge\gamma_m$. By \eqref{eq:Zhao-monotonicity},
\[
 \cos\alpha_1\le\varphi(2,b_M,b_M,\gamma_m,b_M,b_M).
\]
A direct simplification gives
\[
 \varphi(2,p,p,\delta,p,p)=\frac{2p^2-\delta}{2p^2+1}.
\]
Since $\arccos$ is decreasing, the definition \eqref{eq:A2Mm} gives $\alpha_1\ge A_2(M,m)$.
\end{proof}

Tables~\ref{tab:C2-simple} and~\ref{tab:C2-refined} are valence-only consequences of \Cref{lem:C2-symmetric}. Their column headed $M$ specifies a lower bound on all 4 adjacent quotient-edge valences. No additional restriction means $M=6$, the global minimum in this section.

\begin{table}[htbp]
\centering
\caption{Simple uniform sufficient conditions at $C=2$.}
\label{tab:C2-simple}
\setlength{\tabcolsep}{5pt}
\begin{tabular}{ccc}
\toprule
Target valence $n$ & Required adjacent valence $M$ & Guaranteed angle at each occurrence\\
\midrule
$6$ & $\ge19$ & $>\pi/3$\\
$7$ & $\ge14$ & $>2\pi/7$\\
$8$ & $\ge11$ & $>\pi/4$\\
$9$ & $\ge10$ & $>2\pi/9$\\
\bottomrule
\end{tabular}
\end{table}

\begin{table}[htbp]
\centering
\caption{Refined $C=2$ conditions using the opposite-edge valence.}
\label{tab:C2-refined}
\begin{tabular}{ccc}
\toprule
Target valence $n$ & Opposite-edge valence $m$ & Required adjacent valence $M$\\
\midrule
$6$ & $m=6$ & $\ge15$\\
$6$ & $m=7$ & $\ge16$\\
$6$ & $m=8,9$ & $\ge17$\\
$6$ & $10\le m\le15$ & $\ge18$\\
$6$ & $m\ge16$ & $\ge19$\\
$7$ & $m=6,7$ & $\ge12$\\
$7$ & $8\le m\le18$ & $\ge13$\\
$7$ & $m\ge19$ & $\ge14$\\
$8$ & $m=6,7$ & $\ge10$\\
$8$ & $m\ge8$ & $\ge11$\\
$9$ & $6\le m\le11$ & no additional restriction\\
$9$ & $m\ge12$ & $\ge10$\\
\bottomrule
\end{tabular}
\end{table}

\begin{corollary}\label{cor:C2-explicit}
Assume $v(e)\ge6$ for every edge $e\in E$. If every local occurrence of every edge of valence $6,7,8$, or $9$ satisfies either the corresponding row of Table~\ref{tab:C2-simple} or the applicable row of Table~\ref{tab:C2-refined}, then $(M,\mathcal T)$ is geometric.
\end{corollary}

\begin{proof}
Every listed row gives $A_2(M,m)>2\pi/n$ for the target valence $n$.  Thus every occurrence angle is larger than $2\pi/n$, and \Cref{thm:C2-star} applies.  The strict inequalities were checked with outward-rounded interval arithmetic; the monotonicity of $b_M$ in $M$ and of $\gamma_m$ in $m$ reduces each displayed range to its worst endpoint.  See Appendix~\ref{app:numerics}.
\end{proof}

\begin{corollary}\label{cor:C2-v8}
Assume $v(f)\ge8$ for every edge $f\in E$.  Suppose that each local occurrence satisfies the following conditions.
\begin{enumerate}[label=(\roman*)]
\item If the target edge has valence 8, all 4 adjacent edges have valence at least 11.
\item If the target edge has valence 9 and the opposite edge has valence between 8 and 11, no further adjacent-edge condition is imposed.
\item If the target edge has valence 9 and the opposite edge has valence at least 12, all 4 adjacent edges have valence at least 10.
\end{enumerate}
Then $(M,\mathcal T)$ is geometric.
\end{corollary}

\begin{proof}
This is the specialization of Table~\ref{tab:C2-refined} to the global assumption $v(f)\ge8$.  Equivalently, use
\[
 A_2(11,\infty)>\frac\pi4,
 \qquad A_2(8,11)>\frac{2\pi}{9},
 \qquad A_2(10,\infty)>\frac{2\pi}{9},
\]
and apply \Cref{thm:C2-star}.
\end{proof}

The paired-face argument in \cite[proof of Lemma~4.1]{Zhao} gives the following observation, used in \Cref{ex:C2-mixed}.

\begin{lemma}\label{lem:adjacency-propagation}
Let $e$ be a quotient edge of valence greater than 1 in an ideal triangulation. For each occurrence $j$ let $M_j$ be the minimum of its 4 adjacent quotient-edge valences. For any integer $r$, the number of occurrences with $M_j\le r$ is not 1.
\end{lemma}
\begin{proof}
If an occurrence $\widehat e$ is adjacent to a local edge $\widehat f$ whose quotient valence is at most $r$, the two edges lie in a common triangular face. Its face pairing sends them to adjacent occurrences of the same quotient edges. The image of $\widehat e$ is a different local occurrence: otherwise its two incident local faces are paired to each other, closing its edge-link interval by itself, and the edge class has valence 1. Thus there is a second occurrence with adjacent minimum at most $r$. The two occurrences may lie in the same abstract tetrahedron; they are still counted separately.
\end{proof}

Consequently, if all 8 adjacent minima are at least 10 and at least 7 are at least 11, then all 8 are at least 11. Allowing a single exception therefore gives no additional triangulations. The triangulation in \Cref{ex:C2-mixed} has 2 occurrences with adjacent minimum 10.

\subsection{Examples}\label{subsec:C2-examples}

\label{subsec:cyclic-family}

All triangulations in this subsection are geometric. The cyclic family is covered by \Cref{cor:C2-explicit}, as proved in \Cref{cor:cyclic-bipyramid-geometric}. The finite examples follow from \Cref{cor:C2-explicit}, \Cref{cor:C2-v8}, and \Cref{thm:C2-star}, respectively.

The following cyclic bipyramid construction, whose face pairings are specified in the proof, satisfies Table~\ref{tab:C2-simple}. It is motivated by the bipyramid triangulations associated with the Paoluzzi--Zimmermann manifolds \cite{PaoluzziZimmermann,FominykhVesnin}; Figure~\ref{fig:cyclic-bipyramid} shows the local subdivision and the outer-face coupling.

\begin{figure}[!htbp]
 \centering
 \includegraphics[width=\textwidth]{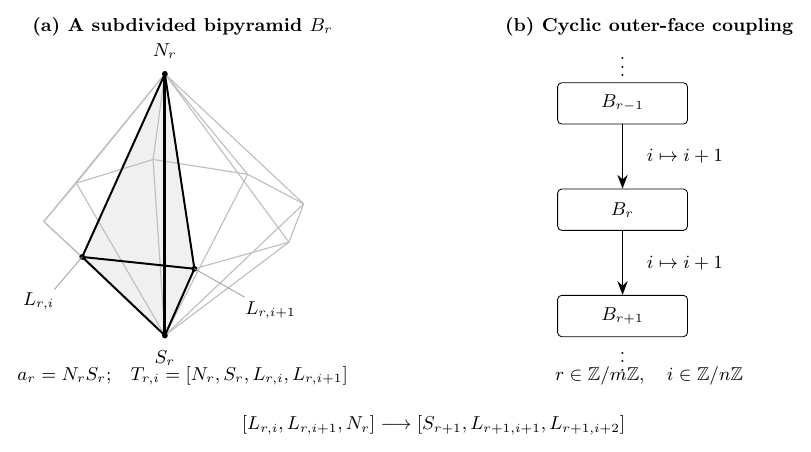}
 \caption{The cyclic bipyramid construction.  Each $B_r$ is subdivided into $T_{r,i}=[N_r,S_r,L_{r,i},L_{r,i+1}]$.  The central edge is $a_r=N_rS_r$, while the other 5 local edges lie in a single quotient class $b$.  The displayed outer-face pairing advances both indices.}
 \label{fig:cyclic-bipyramid}
\end{figure}

\begin{proposition}\label{prop:cyclic-bipyramid}
Let $n\ge6$ and $m\ge1$ satisfy $\gcd(n,m)=1$.  There exists a connected orientable one-vertex ideal triangulation $\mathcal T_{n,m}$ with $mn$ tetrahedra and $m+1$ quotient edges
\[
 a_0,\ldots,a_{m-1},b,
\]
such that
\[
 v(a_r)=n\quad(0\le r<m),
 \qquad
 v(b)=5mn.
\]
Every tetrahedron contains one local occurrence of some $a_r$ and five local occurrences of $b$.  The associated compact $3$-manifold has connected boundary of genus $m(n-1)$.
\end{proposition}

\begin{proof}
Take indices $r\in\mathbb Z/m\mathbb Z$ and $i\in\mathbb Z/n\mathbb Z$.  Let
\[
 T_{r,i}=[N_r,S_r,L_{r,i},L_{r,i+1}].
\]
Within a fixed $r$, pair the two occurrences of the face $[N_r,S_r,L_{r,i}]$ in the adjacent tetrahedra $T_{r,i-1}$ and $T_{r,i}$ by the identity.  Pair the remaining faces by
\begin{equation}\label{eq:cyclic-outer-pairing}
 [L_{r,i},L_{r,i+1},N_r]
 \longrightarrow
 [S_{r+1},L_{r+1,i+1},L_{r+1,i+2}]
\end{equation}
in the displayed vertex order.  Every local face occurs in exactly one pair, and the dual graph is connected.

The relations in \eqref{eq:cyclic-outer-pairing} identify all local vertices: as $i$ varies, $L_{r,i}\sim S_{r+1}$; also $L_{r,i+1}\sim L_{r+1,i+1}$ and $N_r\sim L_{r+1,i+2}$.  Thus there is one quotient vertex.

The central edge $a_r=N_rS_r$ is unaffected by the outer pairings and appears once in each of the $n$ tetrahedra with first index $r$.  Hence the $a_r$ are distinct and $v(a_r)=n$.  For the remaining local edges, write
\[
 u_{r,i}=N_rL_{r,i},
 \qquad
 v_{r,i}=S_rL_{r,i},
 \qquad
 h_{r,i}=L_{r,i}L_{r,i+1}.
\]
The outer pairing gives
\begin{equation}\label{eq:cyclic-edge-relations}
 h_{r,i}\sim v_{r+1,i+1},
 \qquad
 u_{r,i+1}\sim h_{r+1,i+1},
 \qquad
 u_{r,i}\sim v_{r+1,i+2}.
\end{equation}
Combining these relations yields
\[
 u_{r,i}\sim u_{r-1,i+1}.
\]
The translation $(r,i)\mapsto(r-1,i+1)$ has a single orbit on $\mathbb Z/m\mathbb Z\times\mathbb Z/n\mathbb Z$ because its orbit length is $\operatorname{lcm}(m,n)=mn$ when $\gcd(n,m)=1$.  Consequently all $u_{r,i}$, and hence all $v_{r,i}$ and $h_{r,i}$ by \eqref{eq:cyclic-edge-relations}, form one quotient edge $b$.  It has five occurrences in each of the $mn$ tetrahedra, so $v(b)=5mn$.

To check the manifold condition, orient every tetrahedron by the displayed order $[N_r,S_r,L_{r,i},L_{r,i+1}]$; the internal and outer pairings reverse the induced face orientations.  Orient
\[
 u_{r,i}:N_r\to L_{r,i},
 \qquad
 v_{r,i}:L_{r,i}\to S_r,
 \qquad
 h_{r,i}:L_{r,i+1}\to L_{r,i}.
\]
The 3 relations in \eqref{eq:cyclic-edge-relations} preserve these orientations, while $N_r\to S_r$ is preserved by the internal pairings.  Hence no quotient edge is identified with itself in reverse.  For each quotient edge, the local edge-link intervals have their endpoints paired by the face pairings; the resulting finite graph is connected by definition of the edge class and is $2$-regular.  It is therefore a circle.  Thus the quotient is a $3$-manifold away from its unique vertex, whose link is a connected closed orientable surface.

Let $X_{n,m}$ denote the face-pairing quotient.  Its cell counts are
\[
 V=1,
 \qquad E=m+1,
 \qquad F=2mn,
 \qquad T=mn,
\]
so
\[
 \chi(X_{n,m})=1-(m+1)+2mn-mn=m(n-1).
\]
If the vertex link is $S_g$, deleting an open cone neighborhood gives a compact orientable $3$-manifold $N_{n,m}$ with $\partial N_{n,m}=S_g$.  Since $X_{n,m}=N_{n,m}\cup_{S_g}C(S_g)$,
\[
 \chi(X_{n,m})
 =\chi(N_{n,m})+1-\chi(S_g)
 =\frac12\chi(S_g)+1-\chi(S_g)=g.
\]
Therefore $g=m(n-1)$, as claimed.
\end{proof}

\begin{corollary}\label{cor:cyclic-bipyramid-geometric}
For every $n\ge6$ and $m\ge1$ with $\gcd(n,m)=1$, the triangulation $\mathcal T_{n,m}$ is geometric and the extended Ricci flow converges exponentially to its unique zero-curvature hyper-ideal metric. In particular, with $q$ ranging over the nonnegative integers, the three families
\[
 \bigl\{\mathcal T_{6,\,6q+1}\mid q\ge0\bigr\},\quad
 \bigl\{\mathcal T_{7,\,7q+1}\mid q\ge0\bigr\},\quad
 \bigl\{\mathcal T_{8,\,8q+1}\mid q\ge0\bigr\}
\]
contain edges of valence 6, 7, and 8, respectively, and hence lie outside the uniform minimum-valence-9 theorem. Within each family, the associated compact manifolds are pairwise nonhomeomorphic.
\end{corollary}

\begin{proof}
For a target $a_r$, all four adjacent edges in every occurrence represent $b$, whose valence is $5mn$. If $n=6,7,8,9$, then $5mn\ge30,35,40,45$, respectively, exceeding the required thresholds $19,14,11,10$ in Table~\ref{tab:C2-simple}. The edge $b$ has valence at least 30. If $n\ge10$, every edge has valence at least 10, so there are no low-valence stars to check in \Cref{thm:C2-star}. Thus \Cref{cor:C2-explicit} applies. For each fixed $n\in\{6,7,8\}$, taking $m=nq+1$ gives $\gcd(n,m)=1$ and boundary genus $(n-1)(nq+1)$, which increases strictly with $q$. This proves the assertions for all three families.
\end{proof}

\begin{unnumberedremark}
The family in \Cref{prop:bichromatic-family} has adjacent occurrences of the same valence-7 edge and satisfies neither Table~\ref{tab:C2-simple} nor Table~\ref{tab:C2-refined}; see \Cref{cor:bichromatic-family}.
\end{unnumberedremark}

We next give three finite examples, referring in each case to its face-pairing data in Appendix~\ref{app:examples}. The lists of quotient-edge valences follow the first-occurrence label order of that appendix.

\begin{example}\label{ex:min6}
The face-pairing $\mathcal E_6$ in Appendix~\ref{data:E6} has 5 tetrahedra, quotient-edge valences $(24,6)$, and vertex-link genus 4. At each of the 6 occurrences of the valence-6 edge all adjacent valences are 24. Since $24\ge19$, the valence-6 row of Table~\ref{tab:C2-simple} holds, and $\mathcal E_6$ is geometric by \Cref{cor:C2-explicit}. The valence-6 edge excludes the local bichromatic and uniform minimum-valence-9 criteria. The example also satisfies the occurrence-wise symmetric conditions of \Cref{cor:C0-explicit}.
\end{example}

\begin{example}\label{ex:C2-refined}
The 8-tetrahedron triangulation $\mathcal C$ in Appendix~\ref{data:C} has valences $(11,9,20,8)$ and vertex-link genus 5. At the valence-8 edge its 8 adjacent minima are $11,11,11,11,11,11,20,20$. At the valence-9 edge all opposite valences belong to $\{8,9,11\}$. Thus $\mathcal C$ is geometric by \Cref{cor:C2-v8}. Three occurrences of the valence-9 edge have an adjacent edge of valence 9, so Table~\ref{tab:C2-simple} does not apply. A tetrahedron meeting a low-valence edge uses more than 2 quotient-edge classes, so the local bichromatic hypothesis also fails.
\end{example}

\begin{example}\label{ex:C2-mixed}
The triangulation $\mathcal R$ in Appendix~\ref{data:R} has 7 tetrahedra, valences $(24,8,10)$, and vertex-link genus 5. Its valence-8 adjacent minima are
\[
 (10,10,24,24,24,24,24,24).
\]
They satisfy the pattern $2\times(M\ge10)+4\times(M\ge11)+2\times(M\ge12)$. The following strict rational lower bounds give an explicit certificate:
\[
 A_2(10,\infty)>\frac{183}{250},\quad
 A_2(11,\infty)>\frac{99}{125},\quad
 A_2(12,\infty)>\frac{211}{250}.
\]
Hence the angle sum is greater than
\[
 2\frac{183}{250}+4\frac{99}{125}+2\frac{211}{250}
 =\frac{158}{25}>\frac{44}{7}>2\pi.
\]
Thus $\mathcal R$ is geometric by \Cref{thm:C2-star}. The valence-8 condition in \Cref{cor:C2-v8} fails at 2 occurrences.
\end{example}

\Needspace{5\baselineskip}
\section{A parameterized invariant box and local criteria}\label{sec:parametric}

We extend the endpoint bounds to $3/2<C\le1+\sqrt2$ and derive symmetric and pair-min criteria, including mixed conditions on entire edge stars.

\subsection{The parameterized invariant-box theorem}\label{subsec:parametric-theorem}

We prove the uniform endpoint estimates, nonemptiness of the boxes, and conditional flow bounds needed for \Cref{thm:parametric-box}, which concludes this subsection.

We now allow a variable upper parameter $C$. Throughout this section,
\begin{equation}\label{eq:C-range}
 \frac32<C\le1+\sqrt2.
\end{equation}
For each integer $n\ge6$, put $c_n=\cos(2\pi/n)$ and define
\begin{equation*}
 \Gamma_n(C)=\frac{C+2-c_n}{C+c_n}
\end{equation*}
and
\begin{equation*}
 B_n(C)=
 \begin{cases}
 C,&6\le n\le9,\\[1mm]
 1+\dfrac{2C^2(1-c_n)}{1+c_n},&n\ge10.
 \end{cases}
\end{equation*}
At $C=2$ these give $\Gamma_n(2)=\gamma_n$ and $B_n(2)=b_n$ from \cite[Eqs.~(3.3)--(3.4)]{Zhao}. For general $C$, they use the endpoint identities
\begin{align}
 \varphi(x,1,1,C,1,1)&=\frac{C+2-Cx}{x+1},\label{eq:lower-endpoint-formula}\\
 \varphi(x,C,C,1,C,C)&=\frac{2C^2-x+1}{2C^2+x-1}.
 \label{eq:upper-endpoint-formula}
\end{align}

Solving \eqref{eq:lower-endpoint-formula}$=c_n$ for $x$ gives $\Gamma_n(C)$. Solving \eqref{eq:upper-endpoint-formula}$=c_n$ gives $1+2C^2(1-c_n)/(1+c_n)$, used as $B_n(C)$ only for $n\ge10$. For $6\le n\le9$ we retain the common upper bound $C$.

We first extend the endpoint inequalities in \cite[Lemma~3.4]{Zhao}. The upper bound uses Feng--Ge--Hua's three-endpoint reduction \cite[Theorem~3.6]{FengGeHua}; the comparisons below establish the stated parameter range.

\begin{lemma}\label{lem:uniform-endpoint}
Let $C$ satisfy \eqref{eq:C-range}.  For every $a,x_2,\ldots,x_6\in[1,C]$,
\begin{align}
 \varphi(a,x_2,x_3,x_4,x_5,x_6)
 &\ge\varphi(a,1,1,C,1,1),\notag
\\
 \varphi(a,x_2,x_3,x_4,x_5,x_6)
 &\le\varphi(a,C,C,1,C,C).
 \label{eq:uniform-upper}
\end{align}
\end{lemma}

\begin{proof}
For the lower estimate, since $\partial\varphi/\partial x_4\le0$, the minimum occurs at $x_4=C$.  The derivative formula of Feng--Ge--Hua \cite[Eq.~(3.2)]{FengGeHua} gives
\[
 \frac{\partial\varphi}{\partial x_2}
 =A_0\bigl((aC+x_2x_5-x_3x_6)x_6+x_3+ax_5+Cx_2\bigr),
 \qquad A_0\ge0.
\]
The expression in parentheses is bounded below by
\[
 -Cx_6^2+(aC+1)x_6+a+2C.
\]
This is concave in $x_6$, so its minimum on $[1,C]$ is attained at an endpoint.  At $x_6=1$ the value is $a(C+1)+C+1>0$, while at $x_6=C$ it is
\[
 a(C^2+1)+3C-C^3
 \ge C^2+1+3C-C^3
 =-(C+1)(C^2-2C-1)\ge0.
\]
The last inequality is exactly the upper restriction $C\le1+\sqrt2$. The identities
\begin{align*}
 \varphi(a,x_2,x_3,C,x_5,x_6)
 &=\varphi(a,x_3,x_2,C,x_6,x_5)\\
 &=\varphi(a,x_5,x_6,C,x_2,x_3)
 =\varphi(a,x_6,x_5,C,x_3,x_2)
\end{align*}
show that each of the other 3 adjacent derivatives equals the $x_2$-derivative at a permuted point of the same cube. Thus all 4 derivatives are nonnegative, and the minimum is attained at $(x_2,x_3,x_4,x_5,x_6)=(1,1,C,1,1)$.

For the upper estimate, apply Feng--Ge--Hua's three-endpoint bound \cite[Theorem~3.6]{FengGeHua}, with their common upper bound equal to $C$. It gives
\[
 \varphi(a,x_2,\ldots,x_6)\le\max\{U,V,W\},
\]
where direct substitution in \eqref{eq:phi} yields
\begin{align*}
 U&=\varphi(a,C,C,1,C,C)=\frac{2C^2-a+1}{2C^2+a-1},\\
 V&=\varphi(a,C,1,1,C,1)
   =\frac{-a^2+aC^2+a+2C+1}{(a+C)^2},\\
 W&=\varphi(a,C,C,1,C,1)
   =\frac{\sqrt{a+1}(C^2+C+1-a)}{(a+C)\sqrt{2C^2+a-1}}.
\end{align*}
We show $V,W\le U$. Set
\begin{align*}
 K(a,C)&=(3C+1)(a-1)+2(C+1)(1+2C-C^2),\\
 J(a,C)&=2C^2(C+1)^2(3-C)\\
 &\quad +(a-1)(C-1)(7C^2+4C+1)+4C(a-1)(C-a).
\end{align*}
Clearing denominators and expanding gives the identities
\begin{align*}
 U-V&=\frac{(C-1)(a-1)K(a,C)}{(a+C)^2(2C^2+a-1)},\\
 U^2-W^2&=\frac{(C-1)(a-1)J(a,C)}{(a+C)^2(2C^2+a-1)^2}.
\end{align*}
All denominators are positive. Since $1\le a\le C\le1+\sqrt2<3$ and
$1+2C-C^2=2-(C-1)^2\ge0$, every term displayed in $K$ and $J$ is nonnegative. Thus $V\le U$ and $W^2\le U^2$. Both $U$ and $W$ are positive, so $W\le U$. The maximum is therefore $U$, proving \eqref{eq:uniform-upper}.
\end{proof}

\begin{lemma}\label{lem:GammaB}
For $C$ satisfying \eqref{eq:C-range} and every $n\ge6$,
\[
 1<\Gamma_n(C)<B_n(C)\le C,
\]
with strict $B_n(C)<C$ for $n\ge10$.
\end{lemma}

\begin{proof}
For $n\ge6$ we have $1/2\le c_n<1$. Hence
\[
 \Gamma_n(C)-1=\frac{2(1-c_n)}{C+c_n}>0,
 \qquad
 C-\Gamma_n(C)=\frac{(C+1)(C-2+c_n)}{C+c_n}>0,
\]
since $C>3/2$. For $n\ge10$, the inequality
$c_n\ge c_{10}=(1+\sqrt5)/4>4/5$ gives
\[
 B_n(C)=1+\frac{2C^2(1-c_n)}{1+c_n}
 <1+\frac{2C^2}{9}<C,
 \qquad
 C-1-\frac{2C^2}{9}=\frac{(2C-3)(3-C)}9>0.
\]
Here $3/2<C\le1+\sqrt2<3$. Finally, for $n\ge10$,
\[
 B_n(C)-\Gamma_n(C)
 =\frac{2(1-c_n)\{C^2(C+c_n)-(1+c_n)\}}
 {(1+c_n)(C+c_n)}>0,
\]
because $C>1$ and $0<c_n<1$. For $6\le n\le9$, the already proved $\Gamma_n(C)<C=B_n(C)$ completes the proof.
\end{proof}

The next lemma extends \Cref{lem:C2-conditional} to \eqref{eq:C-range}, using \Cref{lem:uniform-endpoint}.

\begin{lemma}\label{lem:C-conditional}
Fix $C\in(3/2,1+\sqrt2]$. Assume $v(f)\ge6$ for every $f\in E$. Let $l(t)$ solve \eqref{eq:flow}, put $x_f(t)=\cosh l_f(t)$, and fix $t_0>0$. Suppose
\[
 x_f(t)\le C\quad\text{for every }f\in E\text{ and every }t\in[0,t_0],
\]
and
\[
 \Gamma_{v(f)}(C)<x_f(0)<B_{v(f)}(C)\quad\text{for every }f\in E.
\]
Then, for every $t\in[0,t_0]$,
\begin{align*}
 x_f(t)&\ge\Gamma_{v(f)}(C)
 &&\text{for every }f\in E,\\
 x_f(t)&\le B_{v(f)}(C)
 &&\text{for every }f\in E\text{ with }v(f)\ge10.
\end{align*}
\end{lemma}

\begin{proof}
We use the first-crossing argument of \cite[Propositions~3.5--3.6]{Zhao}. By \Cref{lem:GammaB}, each initial interval $(\Gamma_{v(f)}(C),B_{v(f)}(C))$ is nonempty. The two endpoint functions
\[
 f_C(a)=\frac{C+2-Ca}{a+1},\qquad
 g_C(a)=\frac{2C^2-a+1}{2C^2+a-1}
\]
are strictly decreasing, with derivatives $-2(C+1)/(a+1)^2$ and $-4C^2/(2C^2+a-1)^2$.

If the lower bound is violated at an edge $e\in E$, put $n=v(e)\ge6$, the number of local occurrences of $e$. Choose a reached level $1<a<\Gamma_n(C)$ and let $T$ be its first hitting time from above. Then $x_e(T)=a$ and $\dot x_e(T)\le0$. By \Cref{lem:uniform-endpoint}, every incident cosine at time $T$ is at least $f_C(a)>f_C(\Gamma_n(C))=c_n$. Since $c_n\in(-1,1)$, clamping preserves the conclusion that every extended angle is smaller than $2\pi/n$. There are exactly $n$ such angles, so their sum is smaller than $2\pi$. Thus $\Ktil_e(l(T))>0$, and \eqref{eq:xflow} gives $\dot x_e(T)>0$, a contradiction.

For an edge $e$ of valence $n=v(e)\ge10$, a violation of the upper bound gives a reached level $B_n(C)<a<C$. Let $T$ be its first hitting time from below, so $x_e(T)=a$ and $\dot x_e(T)\ge0$. The same lemma gives every incident cosine at time $T$ at most $g_C(a)<g_C(B_n(C))=c_n$. Hence every extended angle is larger than $2\pi/n$, and the sum of the $n$ angles is larger than $2\pi$. Thus $\Ktil_e(l(T))<0$ and $\dot x_e(T)<0$, again a contradiction. No upper bound sharper than $C$ is claimed for $6\le n\le9$. Thus neither barrier can be crossed; equality at a barrier is allowed. Strict curvature signs are used only at levels strictly beyond the barriers.
\end{proof}

\begin{condition}\label{cond:parametric-star}
For every edge $e$ of valence 6, 7, 8, or 9, assign numbers $x_f$ to the quotient edges $f\in\mathcal E(\St(e))$ satisfying
\[
 \Gamma_{v(f)}(C)\le x_f\le B_{v(f)}(C),
 \qquad x_e=C.
\]
Compute the extended angle at every local occurrence of $e$ using the 6 corresponding values.  We require that, for every such assignment,
\begin{equation*}
 \sum_{\widehat e\in\PE^{-1}(e)}\alpha(\widehat e)>2\pi.
\end{equation*}
This condition is static: it is an inequality on a finite-dimensional box and contains no time variable.
\end{condition}

We now apply the first-contact argument used in \Cref{thm:C2-star} at the level $C$.

\begin{theorem}\label{thm:parametric-box}\label{thm:intro-parametric}
Let $C\in(3/2,1+\sqrt2]$.  Assume $v(e)\ge6$ for every edge $e\in E$ and that \Cref{cond:parametric-star} holds. Then every flow whose initial cosh-length vector $x(0)$ lies in the interior of
\[
 Q_C=\prod_{e\in E}[\Gamma_{v(e)}(C),B_{v(e)}(C)]
\]
remains in $Q_C$ for all time.  In particular, $(M,\mathcal T)$ is geometric, the zero-curvature metric is unique, and the extended combinatorial Ricci flow converges exponentially from any positive initial condition.
\end{theorem}

\begin{proof}
Suppose that some edge first reaches $C$, and let $T$ be the first such time.  On $[0,T]$ the hypothesis of \Cref{lem:C-conditional} holds.  Hence at time $T$ every edge satisfies the lower bound, and every edge of valence at least 10 satisfies $x_f(T)\le B_{v(f)}(C)<C$.  Choose a first-contact edge $e$, so $x_e(T)=C$. Its valence is 6, 7, 8, or 9, and the entire edge star lies in the box appearing in \Cref{cond:parametric-star}.  Its incident angle sum is therefore larger than $2\pi$, so $\Ktil_e<0$ and $\dot x_e(T)<0$, a contradiction.

Thus $x_f(t)<C$ for every $f\in E$ and $t\ge0$.  Applying \Cref{lem:C-conditional} on arbitrary finite intervals gives all lower bounds and the high-valence upper bounds for all time.  The flow stays in $Q_C$.  The conclusion follows from \Cref{prop:compactness}, the inequality $C<3$, and the rigidity and convergence theorems.
\end{proof}

\subsection{Symmetric and pair-min corollaries}\label{subsec:abstract-local-bounds}

The goal is the unified criterion in \Cref{cor:combined-star}, a computable consequence of \Cref{thm:parametric-box}. We first prove the upper-face monotonicity lemma and the symmetric occurrence estimate, followed by the pair-min lemma needed for the proof of that corollary.

The next lemma extends \cite[Lemma~3.4]{Zhao} on the upper face $x_1=C$. Unlike the $C=2$ result, it does not assert adjacent-variable monotonicity on the entire cube.

\begin{lemma}\label{lem:upper-face-monotonicity}
Let $1<C\le1+\sqrt2$ and fix $x_1=C$.  On $[1,C]^5$, the function $\varphi$ is nondecreasing in $x_2,x_3,x_5,x_6$ and nonincreasing in $x_4$.
\end{lemma}

\begin{proof}
The sign $\partial\varphi/\partial x_4\le0$ follows directly from \eqref{eq:phi}.  The derivative formula of Feng--Ge--Hua \cite[Eq.~(3.2)]{FengGeHua} gives
\[
 \frac{\partial\varphi}{\partial x_2}
 =A_0\bigl((Cx_4+x_2x_5-x_3x_6)x_6+x_3+Cx_5+x_2x_4\bigr),
 \qquad A_0\ge0.
\]
The bracket is bounded below by
\[
 (C+1-Cx_6)x_6+2C+1.
\]
This is concave in $x_6$. At $x_6=1$ its value is $2C+2>0$, and at $x_6=C$ its value is
\[
 -C^3+C^2+3C+1=-(C+1)(C^2-2C-1)\ge0.
\]
The remaining adjacent variables follow from the symmetries of $\varphi$.
\end{proof}

For $M\ge6$ and $m\in\{6,7,\ldots\}\cup\{\infty\}$, set $\Gamma_\infty(C)=1$ and
\begin{equation*}
 A_C(M,m)=\arccos\!\left(
 \varphi\bigl(C,B_M(C),B_M(C),\Gamma_m(C),B_M(C),B_M(C)\bigr)
 \right).
\end{equation*}
Here $m=\infty$ means that no opposite-valence information is used. When only the common adjacent upper bound $C$ is used, write
\begin{align*}
 A_C(\un,m)&:=\arccos\varphi(C,C,C,\Gamma_m(C),C,C)\\
 &=A_C(6,m)=A_C(7,m)=A_C(8,m)=A_C(9,m),
\end{align*}
since $B_n(C)=C$ for $6\le n\le9$.

No clamp is needed in this definition. Indeed, writing $p=B_M(C)$ and $\delta=\Gamma_m(C)$, direct substitution yields
\begin{equation}\label{eq:symmetric-q}
 q_C(p,\delta):=\varphi(C,p,p,\delta,p,p)=\frac{2p^2-(C-1)\delta}{2p^2+C-1}.
\end{equation}
For $3/2<C\le1+\sqrt2$ and $1\le p,\delta\le C$, its denominator is positive, and
\[
 1-q_C=\frac{(C-1)(1+\delta)}{2p^2+C-1}>0,\qquad
 1+q_C=\frac{4p^2-(C-1)(\delta-1)}{2p^2+C-1}>0,
\]
since the latter numerator is at least $4-(C-1)^2\ge2$. Thus $-1<q_C<1$, also when $m=\infty$.

\begin{lemma}\label{lem:occurrence-parametric}
Fix $C\in(3/2,1+\sqrt2]$ and a local occurrence $e_1$ of a target quotient edge. Suppose that every quotient edge $f$ appearing in its tetrahedron has $v(f)\ge6$ and satisfies
\[
 \Gamma_{v(f)}(C)\le x_f\le B_{v(f)}(C),
\]
with $x_1=C$. If all 4 adjacent quotient edges have valence at least $M$ and the opposite edge has valence $m$, then $\alpha_1\ge A_C(M,m)$. Without opposite-valence information the bound is $A_C(M,\infty)$.
\end{lemma}

\begin{proof}
If $n_i$ is the valence of the quotient edge represented by $e_i$, then $n_i\ge M$ for $i=2,3,5,6$. The function $B_n(C)$ is nonincreasing in $n$ by its formula and \Cref{lem:GammaB}. Hence
\[
 x_i\le B_{n_i}(C)\le B_M(C)\quad(i=2,3,5,6),
 \qquad x_4\ge\Gamma_m(C).
\]
Apply \Cref{lem:upper-face-monotonicity}, followed by \eqref{eq:angle-lower-principle}. The calculation preceding the lemma removes the clamp at the comparison point. These length inequalities use box membership, not valence alone.
\end{proof}

For later comparison, the \emph{symmetric star-sum test} is the condition
\begin{equation}\label{eq:symmetric-star-test}
 \sum_j A_C(M_j,m_j)>2\pi
\end{equation}
for every edge of valence 6, 7, 8, or 9 at one common $C$, where $M_j$ is a lower bound on all 4 adjacent valences and $m_j$ is the opposite valence. By \Cref{lem:occurrence-parametric}, this implies \Cref{cond:parametric-star}, hence geometricity by \Cref{thm:parametric-box}. It is a special case of \Cref{cor:combined-star}.

The following estimate extends the opposite-pair comparison in \cite[Eq.~(4.3) and proof of Lemma~4.2]{Zhao}, using \Cref{lem:upper-face-monotonicity}. Here $C$ is the common upper endpoint, $\delta$ is the opposite lower bound, and the pairwise bound $q=B_R(C)$ is determined by adjacent valences.

Fix $C\in(3/2,1+\sqrt2]$ and $q,\delta\in[1,C]$. For an edge orientation $(e_1,\ldots,e_6)$ and an assignment in $[1,C]^6$, suppose $x_1=C$, $x_4\ge\delta$, and
\begin{equation}\label{eq:pair-min-assumption}
 \min\{x_2,x_5\}\le q,
 \qquad
 \min\{x_3,x_6\}\le q.
\end{equation}
Define
\begin{align*}
 \Phi_{\mathrm{same}}(C,q,\delta)
 &=\varphi(C,q,q,\delta,C,C),\\
 \Phi_{\mathrm{cross}}(C,q,\delta)
 &=\varphi(C,q,C,\delta,C,q),\\
 \Phi_{\mathrm{pair}}(C,q,\delta)
 &=\max\{\Phi_{\mathrm{same}},\Phi_{\mathrm{cross}}\}.
\end{align*}
The two explicit formulas are
\begin{align}
 \Phi_{\mathrm{same}}
 &=\frac{-C^2\delta+2C^2q+C^2+\delta+q^2}
 {(q+1)(2C^2+q-1)},\label{eq:Phi-same}\\
 \Phi_{\mathrm{cross}}
 &=\frac{-C\delta+2Cq+\delta}
 {\sqrt{C+1}\sqrt{2C-1}\sqrt{C+2q^2-1}}.
 \label{eq:Phi-cross}
\end{align}

\begin{lemma}\label{lem:pair-min}
Under all the preceding domain assumptions, if \eqref{eq:pair-min-assumption} holds, then
\[
 \alpha_1\ge\arccos\bigl(\clamp_{[-1,1]}
 \Phi_{\mathrm{pair}}(C,q,\delta)\bigr).
\]
\end{lemma}

\begin{proof}
The feasible set is the union of the 4 boxes obtained by choosing one $q$-bounded edge from each pair $(e_2,e_5)$ and $(e_3,e_6)$. By \Cref{lem:upper-face-monotonicity}, their cosine maxima occur at
\[
 (C,q,q,\delta,C,C),\quad(C,q,C,\delta,C,q),\quad
 (C,C,q,\delta,q,C),\quad(C,C,C,\delta,q,q).
\]
The symmetries of $\varphi$ interchange the first and fourth points and the second and third points. Thus only \eqref{eq:Phi-same} and \eqref{eq:Phi-cross} remain. Taking the larger cosine gives a valid upper bound on all 4 boxes, and \eqref{eq:angle-lower-principle} proves the assertion. This argument does not depend on the target valence.
\end{proof}

Define the pairwise lower bound
\begin{equation*}
 \Theta_C(R,m)=\arccos\!\left(\clamp_{[-1,1]}
 \Phi_{\mathrm{pair}}(C,B_R(C),\Gamma_m(C))\right).
\end{equation*}
For a target occurrence with adjacent valences $n_2,n_3,n_5,n_6$ and opposite valence $m$, put
\[
 M=\min\{n_2,n_3,n_5,n_6\},\qquad
 R=\min\{\max(n_2,n_5),\max(n_3,n_6)\}.
\]
The first number is the strongest common bound on all adjacent valences; the second is the strongest threshold attained at least once in each adjacent opposite-edge pair. Within $Q_C$, an edge of valence $n\ge R$ has $x\le B_n(C)\le B_R(C)$. Thus \Cref{lem:pair-min} gives $\alpha_1\ge\Theta_C(R,m)$.

\begin{corollary}\label{cor:combined-star}\label{cor:parametric-star-sum}\label{cor:pair-min}
Assume $v(e)\ge6$ for every quotient edge $e\in E$. Suppose that one common $C\in(3/2,1+\sqrt2]$ has the following property. For every $e\in E$ of valence 6, 7, 8, or 9 and every occurrence $j$ of $e$, form $M_j,R_j,m_j$ as above, and assume
\begin{equation}\label{eq:combined-star}
 \sum_j\max\{A_C(M_j,m_j),\Theta_C(R_j,m_j)\}>2\pi.
\end{equation}
Then $(M,\mathcal T)$ is geometric. Its zero-curvature hyper-ideal metric is unique, and the extended combinatorial Ricci flow converges to it exponentially from any $l(0)\in\R_{>0}^E$.
\end{corollary}
\begin{proof}
For every admissible assignment on the upper face of the star box, both \Cref{lem:occurrence-parametric} and \Cref{lem:pair-min} apply. Therefore each actual angle is at least the maximum in \eqref{eq:combined-star}. Summation proves \Cref{cond:parametric-star}, and \Cref{thm:parametric-box} applies. The case using only $A_C$ is exactly \eqref{eq:symmetric-star-test}; pair-min estimates may instead be used for any or all occurrences.
\end{proof}

This corollary combines the two kinds of local information without requiring the same type of estimate at every occurrence. It gives a computable sufficient condition for \Cref{thm:parametric-box}.

\subsection{Explicit valence criteria at a fixed parameter}\label{subsec:fixed-parameter-corollaries}

Throughout this subsection, we specialize \Cref{cor:combined-star} to
\[
 C=C_0:=\frac{11}{5}.
\]
The tables below give occurrence-wise and mixed star conditions. \Cref{cor:C0-explicit} combines them in a single criterion at the end of this subsection.

Table~\ref{tab:C0-simple} gives a common adjacent-valence threshold without using the opposite valence. Table~\ref{tab:C0-refined} uses the opposite valence $m$. In both tables, $M$ is a lower bound on all 4 adjacent valences at the target occurrence. ``No additional restriction'' means $M=6$. The same rational parameter $C_0$ must be used at every edge.

\begin{table}[htbp]
\centering
\caption{Uniform non-mixed conditions at $C_0=11/5$, compared with $C=2$. The entries in both threshold columns are lower bounds on all 4 adjacent valences; no opposite data are used.}
\label{tab:C0-simple}
\small
\begin{tabular}{cccc}
\toprule
Target valence $n$ & $M$ at $C=2$ & $M$ at $C_0$ & $A_{C_0}(M,\infty)$\\
\midrule
6 & $\ge19$ & $\ge17$ & $>1.049387>\pi/3$\\
7 & $\ge14$ & $\ge13$ & $>0.907613>2\pi/7$\\
8 & $\ge11$ & $\ge11$ & $>0.799029>\pi/4$\\
9 & $\ge10$ & $\ge10$ & $>0.731697>2\pi/9$\\
\bottomrule
\end{tabular}
\end{table}

\begin{table}[htbp]
\centering
\caption{Refined non-mixed conditions at $C_0=11/5$. The last row for each target valence also covers $m=\infty$. These are the smallest integer adjacent thresholds certified by $A_{C_0}(M,m)>2\pi/n$ on the indicated ranges.}
\label{tab:C0-refined}
\small
\begin{tabular}{ccc}
\toprule
Target valence $n$ & Opposite valence $m$ & Required adjacent valence $M$\\
\midrule
6 & $6\le m\le7$ & $\ge15$\\
6 & $8\le m\le11$ & $\ge16$\\
6 & $m\ge12$ & $\ge17$\\
7 & $6\le m\le8$ & $\ge12$\\
7 & $m\ge9$ & $\ge13$\\
8 & $m=6$ & $\ge10$\\
8 & $m\ge7$ & $\ge11$\\
9 & $6\le m\le10$ & no additional restriction\\
9 & $m\ge11$ & $\ge10$\\
\bottomrule
\end{tabular}
\end{table}

The improvement over $C=2$ in Table~\ref{tab:C0-simple} concerns target valences 6 and 7 and uses no mixed angle sum. It is realized by \Cref{ex:U6,ex:U7}. The parameter $C_0$ is not uniformly better: for example,
\[
 A_{C_0}(6,11)<0.695244<\frac{2\pi}{9}
 <0.699193<A_2(6,11).
\]
Thus a valence-9 occurrence with opposite valence 11 needs an additional adjacent restriction at $C_0$, although it does not at $C=2$. The columns for different parameters cannot be combined edge by edge in a single application of the invariant-box theorem.

Table~\ref{tab:C0-occurrence} gives selected rigorous lower bounds for
\[
 A_{C_0}(M,m)=\arccos\varphi(C_0,B_M(C_0),B_M(C_0),\Gamma_m(C_0),B_M(C_0),B_M(C_0)).
\]
The symbol $m=\infty$ means that no opposite-edge valence is used, so $x_4\ge1$ is the only lower bound.  Each displayed decimal is rounded downward from an outward-rounded interval enclosure.

\begin{table}[!t]
\centering
\caption{Occurrence angle lower bounds at $C_0=11/5$. The row label $M$ is a lower bound on all 4 adjacent valences; $M=6$ means no additional restriction. The opposite valence is $m$. Both panels use the same rows.}
\label{tab:C0-occurrence}
\small
\setlength{\tabcolsep}{5pt}
\renewcommand{\arraystretch}{1.04}
\textit{(a) Smaller opposite valences}\par\smallskip
\begin{tabular}{@{}ccccccc@{}}
\toprule
$M$ & $B_M(C_0)$ & $m=6$ & $7$ & $8$ & $9$ & $10$\\
\midrule
6 & 2.200000 & 0.739861 & 0.722739 & 0.711800 & 0.704381 & 0.699116\\
10 & 2.021944 & 0.800090 & 0.781426 & 0.769507 & 0.761427 & 0.755693\\
11 & 1.834575 & 0.874518 & 0.853895 & 0.840734 & 0.831815 & 0.825490\\
12 & 1.694992 & 0.939128 & 0.916750 & 0.902479 & 0.892813 & 0.885958\\
13 & 1.588072 & 0.995065 & 0.971123 & 0.955865 & 0.945534 & 0.938210\\
14 & 1.504280 & 1.043479 & 1.018147 & 1.002012 & 0.991092 & 0.983353\\
15 & 1.437345 & 1.085437 & 1.058869 & 1.041957 & 1.030515 & 1.022408\\
16 & 1.383000 & 1.121886 & 1.094220 & 1.076618 & 1.064714 & 1.056281\\
17 & 1.338255 & 1.153643 & 1.125001 & 1.106788 & 1.094474 & 1.085752\\
18 & 1.300962 & 1.181410 & 1.151898 & 1.133140 & 1.120461 & 1.111484\\
\bottomrule
\end{tabular}
\par\medskip
\textit{(b) Larger opposite valences}\par\smallskip
\begin{tabular}{@{}ccccccc@{}}
\toprule
$M$ & $m=11$ & $12$ & $14$ & $21$ & $22$ & $\infty$\\
\midrule
6 & 0.695243 & 0.692311 & 0.688236 & 0.682011 & 0.681571 & 0.677069\\
10 & 0.751477 & 0.748284 & 0.743849 & 0.737074 & 0.736595 & 0.731697\\
11 & 0.820838 & 0.817316 & 0.812426 & 0.804956 & 0.804428 & 0.799029\\
12 & 0.880919 & 0.877105 & 0.871808 & 0.863721 & 0.863149 & 0.857305\\
13 & 0.932827 & 0.928754 & 0.923097 & 0.914462 & 0.913851 & 0.907613\\
14 & 0.977665 & 0.973362 & 0.967387 & 0.958268 & 0.957624 & 0.951037\\
15 & 1.016452 & 1.011945 & 1.005689 & 0.996143 & 0.995468 & 0.988574\\
16 & 1.050086 & 1.045400 & 1.038895 & 1.028971 & 1.028270 & 1.021104\\
17 & 1.079346 & 1.074500 & 1.067775 & 1.057517 & 1.056792 & 1.049387\\
18 & 1.104890 & 1.099903 & 1.092983 & 1.082428 & 1.081682 & 1.074064\\
\bottomrule
\end{tabular}
\end{table}

The table immediately produces mixed edge-star criteria, some of which are displayed in Table~\ref{tab:C0-patterns}.  These are concrete combinatorial hypotheses: one records, occurrence by occurrence, a lower bound for the 4 adjacent valences and optionally the opposite-edge valence.

\begin{table}[!htbp]
\centering
\caption{Sufficient edge-star patterns at $C_0=11/5$. The notation $k\times(M\ge r)$ means $k$ occurrences whose 4 adjacent valences are at least $r$. Unless specified otherwise, $m=\infty$ (no opposite information).}
\label{tab:C0-patterns}
\small
\setlength{\tabcolsep}{4pt}
\renewcommand{\arraystretch}{1.12}
\begin{tabularx}{\textwidth}{@{}c X l@{}}
\toprule
$v(e)$ & Edge-star pattern & Lower bound for the angle sum\\
\midrule
6 & $3\times(M\ge16)$ + $3\times(M\ge18)$ & \begin{tabular}[t]{@{}l@{}}$3(1.021104)+3(1.074064)$\\$=6.285504>2\pi$\end{tabular}\\[2pt]
7 & $3\times(M\ge12)$ + $4\times(M\ge14)$ & \begin{tabular}[t]{@{}l@{}}$3(0.857305)+4(0.951037)$\\$=6.376063>2\pi$\end{tabular}\\[2pt]
8 & $6\times(M\ge10)$ + $2\times(M\ge14)$ & \begin{tabular}[t]{@{}l@{}}$6(0.731697)+2(0.951037)$\\$=6.292256>2\pi$\end{tabular}\\[2pt]
9 & $4\times(M\ge10)$ + $5\times(\mathrm{unrestricted})$ & \begin{tabular}[t]{@{}l@{}}$4(0.731697)+5(0.677069)$\\$=6.312133>2\pi$\end{tabular}\\[2pt]
\bottomrule
\end{tabularx}
\end{table}

These bounds apply on the upper face of $Q_C$ under the hypotheses of \Cref{lem:occurrence-parametric}. The first two rows genuinely use summation: $A_{C_0}(16,\infty)<\pi/3$ and $A_{C_0}(12,\infty)<2\pi/7$, so the weaker occurrences do not individually meet the target threshold. The corresponding coarse sums at $C=2$ satisfy
\[
\begin{aligned}
 3A_2(16,\infty)+3A_2(18,\infty)&<6.048760<2\pi,\\
 3A_2(12,\infty)+4A_2(14,\infty)&<6.236233<2\pi.
\end{aligned}
\]
These comparisons concern the specified lower-bound patterns, not failure of every $C=2$ criterion for a triangulation satisfying them. Neither row requires a unique low-adjacent-valence occurrence, which would be excluded by \Cref{lem:adjacency-propagation}.

Table~\ref{tab:pair-values} records $\Theta_{C_0}(R,m)$, rounded downward. For each target occurrence on the upper face $x_e=C_0$ of $Q_{C_0}$, this is a lower angle bound obtained from its adjacent and opposite valence data. A sufficient one-occurrence condition for target valence $n$ is $\Theta_{C_0}(R,m)>2\pi/n$; the required threshold therefore depends on $n$.

\begin{table}[!htbp]
\centering
\caption{Selected pair-min angle lower bounds $\Theta_{C_0}(R,m)$ at $C_0=11/5$. Each adjacent opposite-edge pair contains an edge of valence at least $R$. Both panels use the same rows; $R=6$ is unrestricted and agrees with the $M=6$ row of Table~\ref{tab:C0-occurrence}.}
\label{tab:pair-values}
\small
\setlength{\tabcolsep}{5pt}
\renewcommand{\arraystretch}{1.04}
\textit{(a) Smaller opposite valences}\par\smallskip
\begin{tabular}{@{}ccccccc@{}}
\toprule
$R$ & $m=6$ & $7$ & $8$ & $9$ & $10$ & $11$\\
\midrule
6 & 0.739861 & 0.722739 & 0.711800 & 0.704381 & 0.699116 & 0.695243\\
10 & 0.769078 & 0.751212 & 0.739800 & 0.732063 & 0.726572 & 0.722533\\
11 & 0.802931 & 0.784193 & 0.772227 & 0.764115 & 0.758360 & 0.754127\\
12 & 0.830560 & 0.811102 & 0.798679 & 0.790258 & 0.784284 & 0.779891\\
13 & 0.853322 & 0.833263 & 0.820460 & 0.811783 & 0.805627 & 0.801101\\
14 & 0.872243 & 0.851680 & 0.838558 & 0.829666 & 0.823358 & 0.818720\\
15 & 0.888105 & 0.867117 & 0.853725 & 0.844651 & 0.838215 & 0.833483\\
16 & 0.901509 & 0.880160 & 0.866539 & 0.857310 & 0.850765 & 0.845953\\
18 & 0.922709 & 0.900782 & 0.886796 & 0.877322 & 0.870603 & 0.865663\\
23 & 0.955478 & 0.932647 & 0.918090 & 0.908232 & 0.901241 & 0.896102\\
30 & 0.978120 & 0.954657 & 0.939700 & 0.929573 & 0.922393 & 0.917115\\
34 & 0.985500 & 0.961828 & 0.946741 & 0.936525 & 0.929282 & 0.923959\\
\bottomrule
\end{tabular}
\par\medskip
\textit{(b) Larger opposite valences}\par\smallskip
\begin{tabular}{@{}cccccc@{}}
\toprule
$R$ & $m=12$ & $14$ & $21$ & $22$ & $\infty$\\
\midrule
6 & 0.692311 & 0.688236 & 0.682011 & 0.681571 & 0.677069\\
10 & 0.719475 & 0.715226 & 0.708736 & 0.708277 & 0.703584\\
11 & 0.750922 & 0.746470 & 0.739668 & 0.739187 & 0.734270\\
12 & 0.776565 & 0.771945 & 0.764887 & 0.764388 & 0.759286\\
13 & 0.797674 & 0.792914 & 0.785644 & 0.785130 & 0.779875\\
14 & 0.815209 & 0.810332 & 0.802884 & 0.802357 & 0.796974\\
15 & 0.829901 & 0.824925 & 0.817327 & 0.816789 & 0.811298\\
16 & 0.842310 & 0.837251 & 0.829525 & 0.828979 & 0.823395\\
18 & 0.861924 & 0.856732 & 0.848803 & 0.848242 & 0.842512\\
23 & 0.892213 & 0.886812 & 0.878565 & 0.877982 & 0.872024\\
30 & 0.913120 & 0.907573 & 0.899105 & 0.898507 & 0.892389\\
34 & 0.919930 & 0.914335 & 0.905795 & 0.905191 & 0.899020\\
\bottomrule
\end{tabular}
\end{table}

Since $B_6(C)=C$, the two pair-min comparison points coincide when $R=6$, giving $\Theta_C(6,m)=A_C(\un,m)$. For fixed $C_0$, $\Theta_{C_0}(R,m)$ is nondecreasing in $R$ and nonincreasing in $m$: increasing $R$ decreases the pairwise upper bound, whereas increasing $m$ decreases the opposite lower bound. Table~\ref{tab:pair-criterion} gives sufficient conditions for targets of valence 7, 8, and 9, using the largest $m$ in each indicated range (or $m=\infty$). The valence-7 rows use selected common thresholds on ranges, not a pointwise-optimal classification for every $m$.

\begin{table}[!htbp]
\centering
\caption{Pair-min sufficient conditions at $C_0=11/5$. Each pair $(e_2,e_5)$ and $(e_3,e_6)$ must contain an edge of valence at least $R$. The final column gives $\Theta_{C_0}(R,m)>d>2\pi/n$ throughout the indicated range.}
\label{tab:pair-criterion}
\small
\setlength{\tabcolsep}{4pt}
\renewcommand{\arraystretch}{1.08}
\begin{tabular}{@{}clcl@{}}
\toprule
Target $n$ & Opposite valence $m$ & Required $R$ & Angle lower bound\\
\midrule
7 & $m=6$ & $\ge16$ & $>0.901509>2\pi/7$\\
7 & $m=7$ & $\ge18$ & $>0.900782>2\pi/7$\\
7 & $8\le m\le10$ & $\ge23$ & $>0.901241>2\pi/7$\\
7 & $11\le m\le21$ & $\ge30$ & $>0.899105>2\pi/7$\\
7 & $m\ge22$ or no information & $\ge34$ & $>0.899020>2\pi/7$\\
\midrule
8 & $m=6$ & $\ge11$ & $>0.802931>\pi/4$\\
8 & $7\le m\le9$ & $\ge12$ & $>0.790258>\pi/4$\\
8 & $10\le m\le21$ & $\ge13$ & $>0.785644>\pi/4$\\
8 & $m\ge22$ or no information & $\ge14$ & $>0.796974>\pi/4$\\
\midrule
9 & $6\le m\le10$ & $6$ (unrestricted) & $>0.699116>2\pi/9$\\
9 & $m\ge11$ or no information & $\ge10$ & $>0.703584>2\pi/9$\\
\bottomrule
\end{tabular}
\end{table}

There is no pure pair-min sufficient row for valence 6 at $C_0$. For every $R\ge6$ and $m\ge6$, also allowing $m=\infty$, one has $B_R(C_0)\ge1$ and $\Gamma_m(C_0)\le\Gamma_6(C_0)=37/27$. By upper-face monotonicity,
\[
\begin{aligned}
 \Phi_{\mathrm{pair}}(C_0,B_R(C_0),\Gamma_m(C_0))
 &\ge\Phi_{\mathrm{same}}(C_0,1,37/27)\\
 &=\frac{577}{1089}>\frac12.
\end{aligned}
\]
Consequently $\Theta_{C_0}(R,m)\le\arccos(577/1089)<\pi/3$. Even a mixed sum of 6 such pair-min bounds is below $2\pi$. This limits the present pair-min estimate, not geometricity; symmetric bounds, or the unified maximum in \Cref{cor:combined-star}, can still certify valence-6 edges.

In the second valence-9 row, replacing the threshold 10 by 9 gives no improvement over the unrestricted adjacent bound, because $B_9(C)=C$.

The choice $C_0=11/5$ allows a valence-7 pair-min criterion without opposite-edge information that is unavailable at $C=2$. Indeed, for every $R\ge6$,
\[
 \Theta_2(R,\infty)\le\arccos\frac58<\frac{2\pi}{7}
 <\Theta_{C_0}(34,\infty).
\]
The first inequality follows from
$\Phi_{\mathrm{pair}}(2,B_R(2),1)\ge\Phi_{\mathrm{same}}(2,1,1)=5/8$,
using $B_R(2)\ge1$ and upper-face monotonicity. The middle comparison follows from the cubic for $\cos(2\pi/7)$ used in Section~\ref{subsec:bichromatic}, which is increasing on $[1/2,\infty)$ and has value $1/64>0$ at $5/8$. The last follows from Table~\ref{tab:pair-criterion}.

Pair-min bounds can also be summed over an entire edge star. Table~\ref{tab:C0-pair-patterns} lists sufficient patterns; the groups in each row partition the occurrences. A value in the weaker group may be below $2\pi/v(e)$, but the total lower bound exceeds $2\pi$. The third row is realized by \Cref{ex:mixed-pair-eight}; no realization is asserted for the other rows.

\begin{table}[!htbp]
\centering\small
\caption{Mixed pair-min sufficient patterns at $C_0=11/5$. The condition $R\ge r$ requires an edge of valence at least $r$ in each of the two adjacent opposite-edge pairs. Unless stated otherwise, no opposite-edge information is used.}
\label{tab:C0-pair-patterns}
\setlength{\tabcolsep}{4pt}
\renewcommand{\arraystretch}{1.12}
\begin{tabularx}{\textwidth}{@{}c X l@{}}
\toprule
$v(e)$ & Edge-star pattern & Strict lower bound for the angle sum\\
\midrule
7 & $2\times(R\ge18)+5\times(R\ge30)$, all with $m\le10$
 & \begin{tabular}[t]{@{}l@{}}$2(0.870603)+5(0.922393)$\\$=6.353171>2\pi$\end{tabular}\\[2pt]
8 & $2\times(R\ge12)+6\times(R\ge14)$
 & \begin{tabular}[t]{@{}l@{}}$2(0.759286)+6(0.796974)$\\$=6.300416>2\pi$\end{tabular}\\[2pt]
8 & All $R\ge13$; at least 4 occurrences have $m\le14$
 & \begin{tabular}[t]{@{}l@{}}$4(0.779875)+4(0.792914)$\\$=6.291156>2\pi$\end{tabular}\\[2pt]
9 & $4\times(\mathrm{unrestricted})+5\times(R\ge11)$
 & \begin{tabular}[t]{@{}l@{}}$4(0.677069)+5(0.734270)$\\$=6.379626>2\pi$\end{tabular}\\
\bottomrule
\end{tabularx}
\end{table}

\Needspace{8\baselineskip}
\begin{corollary}\label{cor:C0-explicit}
Assume $v(e)\ge6$ for every quotient edge $e\in E$. At every edge of valence $n\in\{6,7,8,9\}$, suppose either that each occurrence satisfies a $C_0$ row of Table~\ref{tab:C0-simple}, a row of Table~\ref{tab:C0-refined}, or a pair-min row of Table~\ref{tab:pair-criterion}, or that its star satisfies a mixed pattern in Table~\ref{tab:C0-patterns} or Table~\ref{tab:C0-pair-patterns}. All rows must have target valence $n$; the occurrence-wise rows and types of estimate may vary between occurrences. Then $(M,\mathcal T)$ is geometric. Its zero-curvature hyper-ideal metric is unique, and the extended combinatorial Ricci flow converges to it exponentially from any positive initial length vector.
\end{corollary}
\begin{proof}
Use the common parameter $C_0=11/5$ at every edge. For the occurrence-wise conditions, the symmetric rows in Tables~\ref{tab:C0-simple} and~\ref{tab:C0-refined} give $A_{C_0}(M,m)>2\pi/n$, and the pair-min rows in Table~\ref{tab:pair-criterion} give $\Theta_{C_0}(R,m)>2\pi/n$. Monotonicity reduces each opposite-valence range to its largest endpoint, with $\Gamma_\infty=1$ for an unbounded range. The strict comparisons, and the integer thresholds asserted in Table~\ref{tab:C0-refined}, are verified in Appendix~\ref{app:numerics}. For an unrestricted row, $\Theta_{C_0}(6,m)=A_{C_0}(\un,m)$.

For a mixed star, the downward-rounded values in Tables~\ref{tab:C0-occurrence} and~\ref{tab:pair-values} give the displayed sums in Tables~\ref{tab:C0-patterns} and~\ref{tab:C0-pair-patterns}, all strictly larger than $2\pi$. In the third row of Table~\ref{tab:C0-pair-patterns}, choose four occurrences with $m\le14$ and ignore opposite data at the other four. Thus in either case the sum in \eqref{eq:combined-star} exceeds $2\pi$ at every low-valence edge. Apply \Cref{cor:combined-star}.
\end{proof}

\subsection{Choosing the common parameter \texorpdfstring{$C$}{C}}\label{subsec:parameter-choice}

For a single occurrence with all 4 adjacent valences at least $M\ge10$, and without opposite-edge information, set
\[
 a_M=\frac{2(1-\cos(2\pi/M))}{1+\cos(2\pi/M)},
 \qquad B_M(C)=1+a_MC^2.
\]
By \Cref{lem:GammaB}, $B_M(C)<C$ throughout the allowed interval. The symmetric lower angle is
\[
 A_M(C)=\arccos\varphi(C,B_M(C),B_M(C),1,B_M(C),B_M(C)).
\]
Writing
\[
 r_M(C)=\frac{C-1}{2(1+a_MC^2)^2},
\]
one has $\varphi=(1-r_M)/(1+r_M)$.  Hence maximizing $A_M$ is equivalent to maximizing $r_M$.  For $C>1$, the unique critical point is
\begin{equation*}
 C_M^*=\frac{2+\sqrt{4+3/a_M}}{3},
\end{equation*}
since
\[
 r_M'(C)=\frac{1+4a_MC-3a_MC^2}{2(1+a_MC^2)^3}.
\]
The derivative changes from positive to negative at $C_M^*$. Since $a_M\le a_{10}<2/9<1/4$, we have $C_M^*>2$. Hence the maximum on $(3/2,1+\sqrt2]$ is attained at $\min\{C_M^*,1+\sqrt2\}$; the excluded lower endpoint causes no difficulty. For $6\le M\le9$, the adjacent upper bound is $B_M(C)=C$.

For a mixed star one maximizes a sum of angle bounds with a single common $C$; one cannot maximize each summand at a different parameter and then add them. Numerical 1-dimensional exploration can suggest a useful rational $C$. The proof needs only a rigorous interval check that all required star sums at this fixed $C$ exceed $2\pi$, not a proof that the chosen $C$ is a global maximizer.

\begin{remark}\label{rem:no-uniform-v8}
With no adjacent-edge information, the best symmetric one-occurrence estimate over $1<C<3$ is attained at $C=2$:
\[
 \max_{1<C<3}\arccos\varphi(C,C,C,1,C,C)
 =\arccos\frac79<\frac{2\pi}{9}<\frac\pi4.
\]
Consequently, the present one-step box estimate cannot prove geometricity for all triangulations of minimum valence 8 without additional local information.  This identifies a limitation of this particular bound, not an obstruction to geometricity. A proof covering the uniform valence-8 case would need additional length or angle information, perhaps obtained by repeatedly improving earlier bounds, or a different global argument.
\end{remark}

\subsection{Examples}\label{subsec:parametric-examples}

All triangulations in this subsection are geometric. Lists of quotient-edge valences follow the first-occurrence label order in Appendix~\ref{app:examples}. The fixed-parameter criterion \Cref{cor:C0-explicit} applies to $\mathcal U_6$, $\mathcal U_7$, $\mathcal T_5$, $\mathcal P$, and $\Pmix$, using the indicated occurrence-wise or mixed rows. We apply \Cref{cor:combined-star} to $\mathcal T_{23}$ and $\mathcal W_{23}$ at their specified common parameters. Complete face-pairing data are given in Appendix~\ref{app:examples}.

\begin{example}\label{ex:U6}
The triangulation $\mathcal U_6$ in Appendix~\ref{data:U6} has 7 tetrahedra, quotient-edge valences $(6,17,19)$, and connected vertex link of genus 5. Its unique low-valence edge $a$ has occurrences $T_0^{01},\ldots,T_5^{01}$. All 6 adjacent minima are 17, and every opposite valence is 17 or 19. A single bound, without opposite information, gives
\[
 6A_{11/5}(17,\infty)>6(1.049387)=6.296322>2\pi.
\]
Thus $\mathcal U_6$ is geometric by \Cref{cor:C0-explicit}; this is not a mixed-pattern argument. At $C=2$, Table~\ref{tab:C2-refined} requires adjacent minimum at least 19 for these opposite valences, so both $C=2$ tables fail.

The stronger individual $C=2$ test also fails. Put
\[
 J(r,s,m,u,v)=\arccos\varphi(2,b_r,b_s,\gamma_m,b_u,b_v).
\]
Using the angle-coordinate order $(01,02,03,23,13,12)$, its sum at $a$ is
\begin{align*}
 S_2(a)={}&J(17,19,17,17,17)+J(19,19,17,17,17)\\
 &+J(19,19,19,19,17)+J(19,19,19,17,19)\\
 &+J(19,17,19,17,17)+J(17,17,17,17,17)\\
 &\in[6.224059056270,\,6.224059056271]\subset(-\infty,2\pi).
\end{align*}
Hence \Cref{thm:C2-star} does not certify this triangulation. The local bichromatic theorem does not apply because there is an edge of valence 6.
\end{example}

\begin{example}\label{ex:U7}
The triangulation $\mathcal U_7$ in Appendix~\ref{data:U7} has 11 tetrahedra, quotient-edge valences $(7,13,13,20,13)$, and connected vertex link of genus 7. Its unique low-valence edge $a$ has occurrences $T_0^{01},\ldots,T_6^{01}$. At every one of them, the valence vector in the order $(01,02,03,23,13,12)$ is
\[
 (7,13,13,20,13,13).
\]
Consequently a single uniform bound proves
\[
 7A_{11/5}(13,\infty)>7(0.907613)=6.353291>2\pi,
\]
and \Cref{cor:C0-explicit} makes $\mathcal U_7$ geometric. At $C=2$, both tables require adjacent minimum at least 14 for target valence 7 and opposite valence 20. Moreover, all adjacent valences are equal, so even the individual endpoint test has sum
\[
 7A_2(13,20)\in[6.272962862500,\,6.272962862501]\subset(-\infty,2\pi).
\]
Thus \Cref{thm:C2-star} also fails to certify it. The tetrahedron $T_0$ uses the 4 quotient classes $0,1,2,3$ and meets $a$, violating the local bichromatic condition. This example separates the non-mixed $C_0$ criterion from the individual $C=2$ criterion as well as from both tables and the local bichromatic theorem.
\end{example}

\begin{example}\label{ex:table5}
The triangulation $\mathcal T_5$ in Appendix~\ref{data:T5} has 14 tetrahedra, valences $(52,10,14,8)$, and vertex-link genus 11. Its unique valence-8 edge has 6 occurrences with adjacent minimum 10 and 2 with adjacent minimum 14; every opposite edge has valence 52. Ignoring that opposite information, the valence-8 row of Table~\ref{tab:C0-patterns} gives
\[
 6A_{11/5}(10,\infty)+2A_{11/5}(14,\infty)>6.292256>2\pi.
\]
There are no valence-9 edges, so $\mathcal T_5$ is geometric by \Cref{cor:C0-explicit}. The corresponding symmetric sum at $C=2$ satisfies
\[
 6A_2(10,\infty)+2A_2(14,\infty)<6.244488<2\pi.
\]
Thus this fixed combinatorial pattern distinguishes the two symmetric parameter choices. It fails Tables~\ref{tab:C2-simple} and~\ref{tab:C2-refined} and the local bichromatic hypothesis. This comparison does not assert failure of the stronger individual-edge test \Cref{thm:C2-star} or of a full static-box optimization.
\end{example}

For a further parameter comparison, define
\[
 F(C)=4A_C(\un,\infty)+2A_C(12,\infty)
       +A_C(13,\infty)+A_C(14,\infty).
\]
The certified comparisons are
\begin{equation}\label{eq:nominal-switch}
 F(11/5)<6.281544<2\pi<6.286890<F(23/10).
\end{equation}
The upper and lower bounds use the corresponding endpoints of rigorous 6-decimal enclosures at the fixed rational parameters $11/5$ and $23/10$.

\begin{example}\label{ex:parameter-switch}
The triangulation $\mathcal T_{23}$ in Appendix~\ref{data:T23} has 10 tetrahedra, valences $(13,13,8,14,12)$, and vertex-link genus 6. Its valence-8 adjacent minima are
\[
 (8,8,8,8,12,12,13,14).
\]
It therefore realizes the pattern in \eqref{eq:nominal-switch}, so \Cref{cor:combined-star}, using only its symmetric bounds at $C=23/10$, proves that $\mathcal T_{23}$ is geometric. Its actual opposite valences are 12, 13, or 14. Using those data instead of $m=\infty$, the same corollary also proves geometricity at $C=11/5$. Thus this example separates the coarse patterns, not all information available at $11/5$.

A stronger actual-data comparison is furnished by $\mathcal W_{23}$ in Appendix~\ref{data:W23}: it has 14 tetrahedra, valences $(12,11,15,38,8)$, and vertex-link genus 10. Its valence-8 adjacent minima are $(8,8,8,8,12,12,12,15)$ and all 8 opposite valences are 38. Set
\[
 G(C)=4A_C(8,38)+3A_C(12,38)+A_C(15,38).
\]
Here the certified bounds are
\begin{equation*}
 G(11/5)<6.282964<2\pi<6.287539<G(23/10).
\end{equation*}
Consequently $\mathcal W_{23}$ is geometric by \Cref{cor:combined-star}, using its symmetric bounds at $C=23/10$. The same symmetric test with the actual opposite data fails at $C=11/5$. These two fixed-parameter evaluations suffice for the comparison.
\end{example}

\begin{figure}[!htbp]
\centering
\includegraphics[width=.94\textwidth]{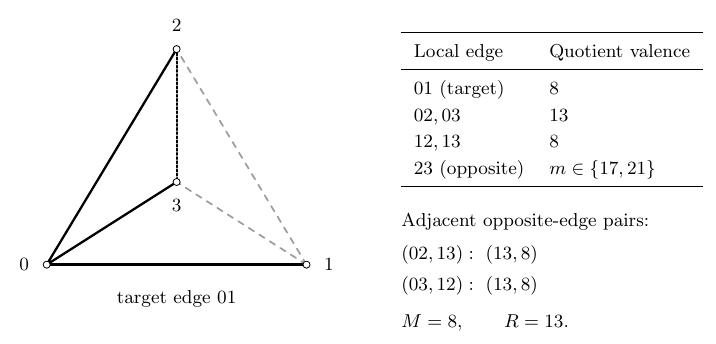}
\caption{A typical occurrence of $a_0$ in $\mathcal P$. The target is $01$. The adjacent opposite pairs are $(02,13)$ and $(03,12)$, each of valences $(13,8)$. The opposite edge $23$ has valence 17 or 21. Each pair provides one stronger upper bound, although the common adjacent minimum is 8.}
\label{fig:pair-witness}
\end{figure}

\begin{example}\label{ex:pair-witness}
The triangulation $\mathcal P$ in Appendix~\ref{data:P} has 16 tetrahedra, valences
\[
 (8,13,13,8,8,21,8,17),
\]
and vertex-link genus 9. Every one of its 32 valence-8 occurrences has pairwise threshold $R\ge13$ and opposite valence $m\le21$. These satisfy the applicable valence-8 rows of Table~\ref{tab:pair-criterion}, and there are no valence-9 edges. Hence $\mathcal P$ is geometric by \Cref{cor:C0-explicit}. For a common strict certificate,
\[
 \Theta_{11/5}(13,21)>0.785644,
 \qquad 8(0.785644)=6.285152>2\pi.
\]
We compare the computable tests on a particular valence-8 edge $a_0$. Its occurrences are $T_0^{01},\ldots,T_7^{01}$. In the order $(01,02,03,23,13,12)$ their valences are
\[
 (8,13,13,m,8,8),\qquad
 m=21\text{ at 6 occurrences and }m=17\text{ at 2}.
\]
Hence both adjacent opposite-edge pairs have valences $(13,8)$, although the minimum over all 4 adjacent edges is only 8; see \Cref{fig:pair-witness}.

The individual $C=2$ endpoint test \eqref{eq:C2-occurrence}, which uses more information than the symmetric tables, has sum
\[
 6L_{21}+2L_{17},\qquad
 L_m=\arccos\varphi(2,b_{13},b_{13},\gamma_m,2,2).
\]
Exact interval evaluation gives
\[
 6L_{21}+2L_{17}\in
 [6.241892741844,\,6.241892741852]\subset(-\infty,2\pi).
\]
Thus the hypothesis of \Cref{thm:C2-star} is not satisfied at this edge.

The symmetric test \eqref{eq:symmetric-star-test} fails for every $C\in(3/2,1+\sqrt2]$. Indeed, all these occurrences have $M=8$ and $m\ge13$, so the symmetric comparison in Appendix~\ref{app:mixed-comparison} gives $A_C(8,m)<\pi/4$ throughout that interval. Hence $6A_C(8,21)+2A_C(8,17)<2\pi$. Finally, $T_0$ uses 6 distinct quotient-edge classes and meets $a_0$, so the local bichromatic condition fails. Thus the pair-min criterion detects $\mathcal P$, whereas the individual $C=2$ criterion, all admissible symmetric star-sum tests, and the local bichromatic criterion do not.
\end{example}

\begin{example}\label{ex:mixed-pair-eight}
Starting with $\mathcal P$ in Appendix~\ref{data:P}, replace the face pairing $T_{11}:F_0\to T_{13}:F_3$ by the permutation \code{3201}, and its inverse by \code{2310}; keep all other face maps. The resulting triangulation $\Pmix$ has 16 tetrahedra and quotient-edge valences
\[
 (8,13,13,8,8,13,8,25)
\]
in first-occurrence label order. Its complete face and link data are given in Appendix~\ref{data:Pmix}; they define a connected orientable one-vertex ideal triangulation with vertex-link genus 9. At the four valence-8 edges, labelled $0,3,4,6$, every occurrence has $R\ge13$, and the numbers with $m\le14$ are respectively $4,5,8,6$. The third row of Table~\ref{tab:C0-pair-patterns} therefore proves geometricity by \Cref{cor:C0-explicit}.

At the edge represented by $T_0^{01},\ldots,T_7^{01}$, the angle-coordinate valence vectors are $(8,13,13,m,8,8)$, with four occurrences each for $m=13$ and $m=25$. In particular,
\begin{gather*}
 \Theta_{C_0}(13,25)<0.783942<\pi/4,\\
 4\Theta_{C_0}(13,13)+4\Theta_{C_0}(13,25)>6.315834>2\pi.
\end{gather*}
The weaker occurrences fail the occurrence-wise pair-min condition, but the mixed sum succeeds. The individual $C=2$ test, all symmetric star-sum tests with $C\in(3/2,1+\sqrt2]$, and the local bichromatic condition also fail. Even after varying $C$, an occurrence-wise test requiring $\max\{A_C,\Theta_C\}>\pi/4$ fails at the weak occurrences. The proofs of these comparisons are in Appendix~\ref{app:mixed-comparison}. These failures distinguish the concrete tests, not their parent invariant-box criterion, which the mixed sum verifies.
\end{example}

\section{Application to compact manifolds with boundary}\label{sec:manifold}

\begin{corollary}\label{cor:boundary-manifold}
Let $N$ be a connected compact $3$-manifold with nonempty boundary, all of whose components have genus at least 2, and let $\mathcal T_N$ be an ideal triangulation. If $(C(N),\mathcal T_N)$ satisfies any one of \Cref{thm:local-bichromatic,thm:C2-star,thm:parametric-box,cor:combined-star}, then $N$ carries a complete hyperbolic metric with totally geodesic boundary, unique up to isometry. The prescribed triangulation is geometrically realized by hyper-ideal tetrahedra, with the same face identifications.
\end{corollary}

\begin{proof}
The applicable criterion gives a genuine zero-curvature hyper-ideal metric. By \Cref{prop:geometric-realization}, the prescribed face identifications give a complete hyperbolic metric on $N$ with totally geodesic boundary, geometrically realizing $\mathcal T_N$; see also \cite[Theorem~2.13 and Remark~2.14]{FrigerioPetronio}. Uniqueness up to isometry is the standard consequence of Mostow rigidity for such manifolds; see \cite[Theorem~2.15]{FrigerioPetronio}, whose doubling argument also applies in the nonorientable case.
\end{proof}

\begin{proposition}\label{prop:covers}
Each finite triangulation $\mathcal E_6,\mathcal C,\mathcal R,\mathcal T_5,\mathcal T_{23},\mathcal W_{23},\mathcal P,\mathcal U_6,\mathcal U_7,\Pmix$ gives geometric lifted triangulations on connected cyclic covers of every positive degree. Within each family, the compact manifolds are pairwise nonhomeomorphic. The local valence and pairwise conditions used to certify the base example are preserved.
\end{proposition}

\begin{proof}
Write $N$ for the compact manifold of a base example and $S_g$ for its connected boundary. Here $g\ge2$. Doubling along the boundary and using Poincar\'e duality gives $\chi(N)=\tfrac12\chi(S_g)=1-g$. Write $b_i(N)=\dim_{\Q}H_i(N;\Q)$. Connectedness and nonempty boundary give $b_0(N)=1$ and $H_3(N;\Q)\cong H^0(N,\partial N;\Q)=0$. Hence $b_1(N)-b_2(N)=g>0$. The finite truncated triangulation gives $N$ a finite CW structure, so $H_1(N;\Z)$ is finitely generated, with positive free rank $b_1(N)$. Abelianization followed by projection onto a free summand gives an epimorphism $\pi_1(N)\twoheadrightarrow\Z$.

For $d\ge1$, compose this epimorphism with reduction modulo $d$ and let $K_d$ be its kernel. Then $K_d$ is normal of index $d$, so $N_d:=\widetilde N/K_d$ is a connected $d$-sheeted regular cover of $N$ with deck group $\Z/d\Z$. Lift the truncated triangulation to $N_d$. Each truncated ideal edge is an interval and has $d$ lifts. The absence of reversed edge identifications keeps its two endpoints distinct. Each lift has the same valence and local opposite-pair data as the base edge, so all numerical star-sum tests take the same values. In particular, the inequalities proving geometricity are preserved.

Finally, $\chi(N_d)=d(1-g)$ distinguishes different degrees. This argument does not require $\partial N_d$ to be connected.
\end{proof}

\appendix

\section{Rigorous evaluation of the numerical tables}\label{app:numerics}

The entries of Tables~\ref{tab:C0-occurrence}, \ref{tab:C0-patterns}, and~\ref{tab:pair-values}, and the strict comparisons in Tables~\ref{tab:C2-simple}, \ref{tab:C2-refined}, \ref{tab:C0-simple}, \ref{tab:C0-refined}, and~\ref{tab:pair-criterion}, are certified by rational interval arithmetic. The numerical inequalities in the examples are checked by the same method. We specify the finite verification procedure here.

For a fixed integer $n$, an interval enclosure for $c_n=\cos(2\pi/n)$ is first computed at high precision.  The rational operations and square roots in $\Gamma_n(C)$, $B_n(C)$, $\varphi$, $\Phi_{\mathrm{same}}$, and $\Phi_{\mathrm{cross}}$ are then evaluated with outward rounding.  If an interval $I=[a,b]\subset[-1,1]$ encloses a cosine expression, monotonicity of $\arccos$ gives the angle enclosure
\[
 [\arccos b,\arccos a].
\]
For each angle entry $d$ printed to 6 decimal places, the checks prove $d<\theta<d+10^{-6}$. More explicitly, $\pi$ is enclosed by Machin's identity
\[
 \pi=16\arctan(1/5)-4\arctan(1/239),
\]
using consecutive partial sums of the alternating arctangent series. For rational $x$ with $x^2<2$, consecutive partial sums of $\cos x=\sum_{k\ge0}(-1)^kx^{2k}/(2k)!$ give rational enclosures; the ratio of successive absolute terms is $x^2/((2k+1)(2k+2))<1$. If $q$ is the cosine argument, the two certified comparisons $q<\cos d$ and $q>\cos(d+10^{-6})$ prove the asserted rounding. Square roots are enclosed by rational numbers whose squared values bracket the argument. All these tests use exact integer or rational comparisons.

For the $C=2$ refined table, $M\mapsto b_M$ is nonincreasing and $m\mapsto\gamma_m$ is decreasing.  Consequently $A_2(M,m)$ is nondecreasing in $M$ and decreasing in $m$.  Each range in Table~\ref{tab:C2-refined} therefore reduces to its largest allowed $m$, or to the limiting value $m=\infty$.  The closest positive comparisons include
\[
\begin{gathered}
 A_2(18,15)>\frac\pi3,\qquad A_2(13,18)>\frac{2\pi}{7},\qquad A_2(10,7)>\frac\pi4,\\
 A_2(6,11)=A_2(8,11)>\frac{2\pi}{9}.
\end{gathered}
\]
The neighbouring failed comparisons determine the displayed break points.

For the non-mixed $C_0$ tables, clearing the positive denominator in \eqref{eq:symmetric-q} gives the equivalent rational-interval sign test
\[
 A_C(M,m)>\frac{2\pi}{n}
 \quad\Longleftrightarrow\quad
 (C-1)\bigl(\Gamma_m(C)+c_n\bigr)
       -2(1-c_n)B_M(C)^2>0.
\]
Both $B_M(C)$ and $\Gamma_m(C)$ decrease with their integer indices (with $B_M$ constant for $6\le M\le9$). Hence $A_C$ is nondecreasing in $M$ and decreasing in $m$. For a row $a\le m\le b$ with adjacent threshold $M$, we certify the displayed threshold at $(M,b)$ and, when $M>6$, its failure at $(M-1,a)$. This proves the entire range, including $b=\infty$, without searching an unbounded set of integers. The four uniform thresholds in Table~\ref{tab:C0-simple} are verified similarly with $m=\infty$.

For \Cref{ex:U6,ex:U7}, the full individual $C=2$ endpoint sums are evaluated at every occurrence using the valence vectors reconstructed from the face pairings, not by substituting rounded values of $b_n$ or $\gamma_m$. The resulting intervals are the ones displayed in \Cref{ex:U6,ex:U7}; the common $C_0$ lower bounds are $A_{C_0}(17,\infty)>1.049387$ and $A_{C_0}(13,\infty)>0.907613$.

All angle cells of Tables~\ref{tab:C0-occurrence} and~\ref{tab:pair-values} are checked directly from the defining formulas. For a pair-min row with opposite range $a\le m\le b$, it is enough to certify
\[
 \Phi_{\mathrm{pair}}(C_0,B_R(C_0),\Gamma_b(C_0))<c_n,
\]
with $\Gamma_\infty=1$ when $b=\infty$. Both $\Phi_{\mathrm{same}}$ and $\Phi_{\mathrm{cross}}$ are evaluated; neither is assumed in advance to be the larger. The valence-7 rows of Table~\ref{tab:pair-criterion} use the five inputs
\[
 (R,b)=(16,6),\ (18,7),\ (23,10),\ (30,21),\ (34,\infty).
\]
For example, the last input gives $\Theta_{C_0}(34,\infty)>0.899020>2\pi/7$. The valence-6 exclusion is instead proved analytically above, so it does not rely on a finite search over $R$.

For the valence-8 pair-min rows, the closest positive and negative comparisons at the transition are
\[
 \Phi_{\mathrm{pair}}(C_0,B_{13}(C_0),\Gamma_{21}(C_0))
 <\frac{\sqrt2}{2},
\]
whereas
\[
 \Phi_{\mathrm{pair}}(C_0,B_{13}(C_0),\Gamma_{22}(C_0))
 >\frac{\sqrt2}{2}.
\]
Thus the break between $m=21$ and $m=22$ in Table~\ref{tab:pair-criterion} is genuine for this estimate.  The mixed star sums in Tables~\ref{tab:C0-patterns} and~\ref{tab:C0-pair-patterns} are checked by summing the lower endpoints and comparing the result with an upper enclosure of $2\pi$.

Appendix~\ref{app:mixed-comparison} also uses the finite interval comparisons $3/25<2(1-c_{13})/(1+c_{13})<1/4$ and $\Gamma_{25}(3/2)<21/20$. The endpoint value is used only to bound $\Gamma_{25}(C)\le\Gamma_{25}(3/2)$ for $C>3/2$; \Cref{thm:parametric-box} is not applied at $C=3/2$.

After the stated monotonicity and endpoint reductions, these assertions require only interval evaluations at finitely many explicitly specified inputs. No continuous optimization is required for these finite checks. The continuous parameter comparisons in \Cref{ex:pair-witness} and Appendix~\ref{app:mixed-comparison}, in contrast, are proved analytically for the entire parameter interval.

\section{Face-pairing data and finite combinatorial certificates}\label{app:examples}

The records below completely specify the geometric triangulations established in \Cref{subsec:C2-examples,subsec:parametric-examples}; each data set refers to the example that proves its geometricity. The encoding is that of \Cref{subsec:encoding}. In a row for tetrahedron $t$, the 4 targets and the 4 permutations correspond in order to faces $F_0,F_1,F_2,F_3$. The local edge columns have the fixed order $01,02,03,12,13,23$. Edge-class labels are integers local to each example, numbered by first occurrence. The sign after a label records the orientation of the corresponding increasing local edge $i\to j$ relative to a chosen orientation of its quotient edge.

These records can be verified by the following finite procedure. For every face map $p:F_f^t\to F_{p(f)}^u$, check that the reverse record is $(t,p^{-1})$ in face column $p(f)$ of tetrahedron $u$. Form equivalence classes using only the 3 vertices and 3 edges contained in $F_f$, not its opposite vertex. The resulting vertex classes, edge classes and edge occurrence counts give the data below. The column $\varepsilon_t$ specifies the sign of the orientation on $T_t$. For every face map the records satisfy $\varepsilon_u=-\operatorname{sgn}(p)\varepsilon_t$, so all induced face orientations are reversed. The signed edge columns certify that every face map preserves the assigned quotient-edge orientations; therefore no edge is identified with itself in reverse.

The interval transverse to each local edge has one endpoint for each of its 2 incident faces. Pairing these endpoints produces a finite 2-regular graph; each quotient-edge class is connected by its definition, so this graph is a circle. To check the remaining vertex links directly, take one triangle for each local vertex $(t,v)$, with corners labelled by the other vertices $w\ne v$. A face map identifies $(t,v,w)$ with $(u,p(v),p(w))$ whenever $v,w\ne f$. For each resulting vertex, the corner-link intervals form a connected 2-regular graph: their endpoints are paired by the triangle-side maps, and their connectivity follows from the corner equivalence class. Thus the glued triangles form a closed surface. It is connected in every example. Orienting the triangle at $(t,v)$ by $(-1)^{v+1}\varepsilon_t$ relative to the increasing order of the other vertex labels gives opposite boundary orientations along paired sides, so the surface is orientable. If there are $T$ tetrahedra and $E$ quotient edges, its cell counts are $2E,6T,4T$, giving genus $1+T-E$. Hence truncation gives a compact orientable $3$-manifold with the stated boundary.

\begin{center}
\begin{tabular}{cccl}
\toprule
Example&$T$&Boundary genus&Quotient-edge valences (in label order)\\
\midrule
$\mathcal E_6$ & 5 & 4 & $(24, 6)$\\
$\mathcal C$ & 8 & 5 & $(11, 9, 20, 8)$\\
$\mathcal R$ & 7 & 5 & $(24, 8, 10)$\\
$\mathcal T_5$ & 14 & 11 & $(52, 10, 14, 8)$\\
$\mathcal T_{23}$ & 10 & 6 & $(13, 13, 8, 14, 12)$\\
$\mathcal W_{23}$ & 14 & 10 & $(12, 11, 15, 38, 8)$\\
$\mathcal P$ & 16 & 9 & $(8, 13, 13, 8, 8, 21, 8, 17)$\\
 $\mathcal U_6$ & 7 & 5 & $(6, 17, 19)$\\
 $\mathcal U_7$ & 11 & 7 & $(7, 13, 13, 20, 13)$\\
$\Pmix$ & 16 & 9 & $(8,13,13,8,8,13,8,25)$\\
\bottomrule
\end{tabular}
\end{center}

\Needspace{7\baselineskip}
\subsection{\texorpdfstring{Data for $\mathcal E_6$}{Data for E6}}\label{data:E6}
This is the triangulation used in \Cref{ex:min6}. Its edge valences, in label order, are $(24, 6)$.

\begingroup
\small
\setlength{\tabcolsep}{5pt}
\setlength{\LTpre}{6pt}
\setlength{\LTpost}{6pt}
\renewcommand{\LTcaptype}{} 
\stepcounter{LT@tables} 
\begin{longtable}{rclll}
\toprule
$t$ & $\varepsilon_t$ & Face targets & Face permutations & Signed local edge classes\\
\midrule
\endfirsthead
\multicolumn{5}{l}{\textit{Face-pairing data (continued)}}\\
\toprule
$t$ & $\varepsilon_t$ & Face targets & Face permutations & Signed local edge classes\\
\midrule
\endhead
\midrule
\multicolumn{5}{r}{\textit{Continued on next page}}\\
\endfoot
\bottomrule
\endlastfoot
0 & + & \code{3 1 4 1} & \code{3201 0123 1203 1032} & \code{0+ 0- 0- 0- 1+ 0-}\\
1 & - & \code{2 0 0 2} & \code{3120 0123 1032 2031} & \code{0- 0- 0- 0- 0- 0-}\\
2 & - & \code{3 1 4 1} & \code{1320 1302 0213 3120} & \code{0+ 0+ 0- 0- 1+ 0-}\\
3 & + & \code{4 2 4 0} & \code{3102 3021 1320 2310} & \code{0- 0+ 1- 1- 0+ 0+}\\
4 & - & \code{0 2 3 3} & \code{2013 0213 3021 2130} & \code{1+ 0+ 0- 0+ 0- 1+}\\
\end{longtable}
\endgroup

For each low-valence class $e$, the following records $(M,m,R)$ give the common adjacent minimum, opposite valence, and pairwise threshold, with multiplicities:
\begin{center}
\small
\begin{tabular}{cl}
\toprule
Class $e$ & Occurrence data\\
\midrule
1 & $4\times(24,6,24),\quad 2\times(24,24,24)$\\
\bottomrule
\end{tabular}
\end{center}

\Needspace{7\baselineskip}
\subsection{\texorpdfstring{Data for $\mathcal C$}{Data for C}}\label{data:C}
This is the triangulation used in \Cref{ex:C2-refined}. Its edge valences, in label order, are $(11, 9, 20, 8)$.

\begingroup
\small
\setlength{\tabcolsep}{5pt}
\setlength{\LTpre}{6pt}
\setlength{\LTpost}{6pt}
\renewcommand{\LTcaptype}{} 
\stepcounter{LT@tables} 
\begin{longtable}{rclll}
\toprule
$t$ & $\varepsilon_t$ & Face targets & Face permutations & Signed local edge classes\\
\midrule
\endfirsthead
\multicolumn{5}{l}{\textit{Face-pairing data (continued)}}\\
\toprule
$t$ & $\varepsilon_t$ & Face targets & Face permutations & Signed local edge classes\\
\midrule
\endhead
\midrule
\multicolumn{5}{r}{\textit{Continued on next page}}\\
\endfoot
\bottomrule
\endlastfoot
0 & + & \code{0 1 0 4} & \code{2310 2031 3201 2310} & \code{0+ 0- 1+ 1+ 1- 0-}\\
1 & + & \code{0 4 5 6} & \code{1302 0132 0132 2031} & \code{2+ 2- 3+ 1- 0+ 0-}\\
2 & + & \code{5 3 6 7} & \code{0132 0213 2031 2031} & \code{2+ 2- 1- 3+ 0+ 2+}\\
3 & + & \code{7 6 2 4} & \code{2031 0213 0213 1023} & \code{2- 0+ 1- 3+ 2+ 2+}\\
4 & + & \code{0 1 6 3} & \code{3201 0132 3201 1023} & \code{2+ 3+ 2- 0+ 1- 0+}\\
5 & + & \code{2 7 7 1} & \code{0132 1023 0132 0132} & \code{2+ 3+ 2- 0+ 3+ 2-}\\
6 & + & \code{4 1 3 2} & \code{2310 1302 0213 1302} & \code{0+ 2- 1- 1+ 2+ 2-}\\
7 & + & \code{5 2 3 5} & \code{1023 1302 1302 0132} & \code{2+ 2- 3+ 3+ 2- 2-}\\
\end{longtable}
\endgroup

For each low-valence class $e$, the following records $(M,m,R)$ give the common adjacent minimum, opposite valence, and pairwise threshold, with multiplicities:
\begin{center}
\small
\begin{tabular}{cl}
\toprule
Class $e$ & Occurrence data\\
\midrule
1 & $2\times(9,9,11),\quad 1\times(9,11,9),\quad 4\times(11,8,20)$\\
 & $2\times(11,9,20)$\\
3 & $2\times(11,8,20),\quad 4\times(11,9,20),\quad 2\times(20,8,20)$\\
\bottomrule
\end{tabular}
\end{center}

\Needspace{7\baselineskip}
\subsection{\texorpdfstring{Data for $\mathcal R$}{Data for R}}\label{data:R}
This is the triangulation used in \Cref{ex:C2-mixed}. Its edge valences, in label order, are $(24, 8, 10)$.

\begingroup
\small
\setlength{\tabcolsep}{5pt}
\setlength{\LTpre}{6pt}
\setlength{\LTpost}{6pt}
\renewcommand{\LTcaptype}{} 
\stepcounter{LT@tables} 
\begin{longtable}{rclll}
\toprule
$t$ & $\varepsilon_t$ & Face targets & Face permutations & Signed local edge classes\\
\midrule
\endfirsthead
\multicolumn{5}{l}{\textit{Face-pairing data (continued)}}\\
\toprule
$t$ & $\varepsilon_t$ & Face targets & Face permutations & Signed local edge classes\\
\midrule
\endhead
\midrule
\multicolumn{5}{r}{\textit{Continued on next page}}\\
\endfoot
\bottomrule
\endlastfoot
0 & + & \code{2 5 5 4} & \code{0321 1302 3120 1023} & \code{0+ 1+ 0+ 0- 1+ 0+}\\
1 & + & \code{6 2 6 2} & \code{2310 1302 2031 2031} & \code{2+ 0+ 0- 0- 2+ 2-}\\
2 & + & \code{0 1 3 1} & \code{0321 1302 2103 2031} & \code{0- 2- 0- 0- 1- 0+}\\
3 & + & \code{2 5 6 5} & \code{2103 2103 0213 2310} & \code{2- 0- 2- 0+ 1- 0-}\\
4 & + & \code{4 6 4 0} & \code{2031 3012 1302 1023} & \code{0- 0- 1+ 1+ 0- 0+}\\
5 & + & \code{3 3 0 0} & \code{3201 2103 3120 2031} & \code{1- 0+ 0- 0+ 0- 2-}\\
6 & + & \code{4 3 1 1} & \code{1230 0213 3201 1302} & \code{2+ 2- 2- 0+ 0+ 1-}\\
\end{longtable}
\endgroup

For each low-valence class $e$, the following records $(M,m,R)$ give the common adjacent minimum, opposite valence, and pairwise threshold, with multiplicities:
\begin{center}
\small
\begin{tabular}{cl}
\toprule
Class $e$ & Occurrence data\\
\midrule
1 & $1\times(10,10,24),\quad 1\times(10,24,24),\quad 4\times(24,8,24)$\\
 & $2\times(24,10,24)$\\
\bottomrule
\end{tabular}
\end{center}

\Needspace{7\baselineskip}
\subsection{\texorpdfstring{Data for $\mathcal T_5$}{Data for T5}}\label{data:T5}
This is the triangulation used in \Cref{ex:table5}. Its edge valences, in label order, are $(52, 10, 14, 8)$.

\begingroup
\small
\setlength{\tabcolsep}{5pt}
\setlength{\LTpre}{6pt}
\setlength{\LTpost}{6pt}
\renewcommand{\LTcaptype}{} 
\stepcounter{LT@tables} 
\begin{longtable}{rclll}
\toprule
$t$ & $\varepsilon_t$ & Face targets & Face permutations & Signed local edge classes\\
\midrule
\endfirsthead
\multicolumn{5}{l}{\textit{Face-pairing data (continued)}}\\
\toprule
$t$ & $\varepsilon_t$ & Face targets & Face permutations & Signed local edge classes\\
\midrule
\endhead
\midrule
\multicolumn{5}{r}{\textit{Continued on next page}}\\
\endfoot
\bottomrule
\endlastfoot
0 & + & \code{1 10 12 6} & \code{3201 2103 2310 3120} & \code{0+ 1+ 0+ 2+ 0+ 0+}\\
1 & + & \code{13 5 12 0} & \code{3012 3120 2103 2310} & \code{0+ 2- 0+ 0- 0- 3+}\\
2 & + & \code{9 11 6 4} & \code{2310 1302 2031 0213} & \code{0- 0+ 2- 2+ 0+ 0-}\\
3 & + & \code{12 7 4 5} & \code{2103 3201 2310 0213} & \code{3- 0+ 0+ 1+ 0- 0-}\\
4 & + & \code{8 3 12 2} & \code{1023 3201 2031 0213} & \code{0+ 0- 0+ 2- 1- 3-}\\
5 & + & \code{6 1 11 3} & \code{1230 3120 3012 0213} & \code{0+ 3- 0- 1- 0- 2+}\\
6 & + & \code{0 5 13 2} & \code{3120 3012 1023 1302} & \code{0+ 0+ 2- 2+ 0- 1-}\\
7 & + & \code{7 9 3 7} & \code{3012 0132 2310 1230} & \code{0- 0- 0- 0- 0- 0-}\\
8 & + & \code{9 4 11 10} & \code{0321 1023 3201 1302} & \code{0- 2- 1- 0- 2+ 3-}\\
9 & + & \code{8 7 2 10} & \code{0321 0132 3201 3120} & \code{0+ 0- 0- 3+ 2- 0+}\\
10 & + & \code{9 0 8 13} & \code{3120 2103 2031 3120} & \code{2+ 1- 0+ 3+ 0- 0+}\\
11 & + & \code{8 5 13 2} & \code{2310 1230 0213 2031} & \code{0- 0- 0- 2- 1+ 0+}\\
12 & + & \code{1 0 3 4} & \code{2103 3201 2103 1302} & \code{1- 0- 0- 0- 0- 0+}\\
13 & + & \code{10 11 6 1} & \code{3120 0213 1023 1230} & \code{0- 0- 0- 3+ 2- 1+}\\
\end{longtable}
\endgroup

For each low-valence class $e$, the following records $(M,m,R)$ give the common adjacent minimum, opposite valence, and pairwise threshold, with multiplicities:
\begin{center}
\small
\begin{tabular}{cl}
\toprule
Class $e$ & Occurrence data\\
\midrule
3 & $1\times(10,52,14),\quad 5\times(10,52,52),\quad 2\times(14,52,52)$\\
\bottomrule
\end{tabular}
\end{center}

\Needspace{7\baselineskip}
\subsection{\texorpdfstring{Data for $\mathcal T_{23}$}{Data for T23}}\label{data:T23}
This is the triangulation used in \Cref{ex:parameter-switch}. Its edge valences, in label order, are $(13, 13, 8, 14, 12)$.

\begingroup
\small
\setlength{\tabcolsep}{5pt}
\setlength{\LTpre}{6pt}
\setlength{\LTpost}{6pt}
\renewcommand{\LTcaptype}{} 
\stepcounter{LT@tables} 
\begin{longtable}{rclll}
\toprule
$t$ & $\varepsilon_t$ & Face targets & Face permutations & Signed local edge classes\\
\midrule
\endfirsthead
\multicolumn{5}{l}{\textit{Face-pairing data (continued)}}\\
\toprule
$t$ & $\varepsilon_t$ & Face targets & Face permutations & Signed local edge classes\\
\midrule
\endhead
\midrule
\multicolumn{5}{r}{\textit{Continued on next page}}\\
\endfoot
\bottomrule
\endlastfoot
0 & + & \code{4 5 3 4} & \code{3012 0213 3012 0132} & \code{0+ 1+ 2+ 0+ 3+ 0+}\\
1 & + & \code{9 6 9 5} & \code{1023 3012 2031 0321} & \code{2+ 4+ 3+ 2- 3- 3+}\\
2 & + & \code{6 4 3 6} & \code{3201 3012 1230 1302} & \code{4- 0- 1- 1+ 3- 0+}\\
3 & + & \code{7 0 8 2} & \code{1230 1230 1230 3012} & \code{1+ 3+ 0- 4- 1+ 2-}\\
4 & + & \code{2 6 0 0} & \code{1230 3120 0132 1230} & \code{0+ 3+ 1+ 0+ 0+ 1+}\\
5 & + & \code{7 1 0 8} & \code{2103 0321 0213 2031} & \code{1+ 4+ 2+ 4- 0+ 2+}\\
6 & + & \code{1 4 2 2} & \code{1230 3120 2031 2310} & \code{0+ 1- 1- 3+ 4- 3-}\\
7 & + & \code{8 3 5 8} & \code{2031 3012 2103 2310} & \code{4+ 1- 2+ 4- 0+ 4-}\\
8 & + & \code{7 5 7 3} & \code{3201 1302 1302 3012} & \code{0+ 1- 4- 1+ 4+ 4+}\\
9 & + & \code{9 1 9 1} & \code{2031 1023 1302 1302} & \code{3- 2- 3- 3- 3- 3+}\\
\end{longtable}
\endgroup

For each low-valence class $e$, the following records $(M,m,R)$ give the common adjacent minimum, opposite valence, and pairwise threshold, with multiplicities:
\begin{center}
\small
\begin{tabular}{cl}
\toprule
Class $e$ & Occurrence data\\
\midrule
2 & $1\times(8,12,13),\quad 1\times(8,13,12),\quad 2\times(8,14,14)$\\
 & $1\times(12,12,12),\quad 1\times(12,13,13),\quad 1\times(13,13,13)$\\
 & $1\times(14,14,14)$\\
\bottomrule
\end{tabular}
\end{center}

\Needspace{7\baselineskip}
\subsection{\texorpdfstring{Data for $\mathcal W_{23}$}{Data for W23}}\label{data:W23}
This is the triangulation used in \Cref{ex:parameter-switch}. Its edge valences, in label order, are $(12, 11, 15, 38, 8)$.

\begingroup
\small
\setlength{\tabcolsep}{5pt}
\setlength{\LTpre}{6pt}
\setlength{\LTpost}{6pt}
\renewcommand{\LTcaptype}{} 
\stepcounter{LT@tables} 
\begin{longtable}{rclll}
\toprule
$t$ & $\varepsilon_t$ & Face targets & Face permutations & Signed local edge classes\\
\midrule
\endfirsthead
\multicolumn{5}{l}{\textit{Face-pairing data (continued)}}\\
\toprule
$t$ & $\varepsilon_t$ & Face targets & Face permutations & Signed local edge classes\\
\midrule
\endhead
\midrule
\multicolumn{5}{r}{\textit{Continued on next page}}\\
\endfoot
\bottomrule
\endlastfoot
0 & + & \code{2 6 4 9} & \code{3201 0321 0132 0132} & \code{0+ 1+ 2+ 3+ 0+ 3+}\\
1 & + & \code{9 1 10 1} & \code{1230 1302 2031 2031} & \code{1+ 1- 3- 3- 1+ 1-}\\
2 & + & \code{11 12 5 0} & \code{1302 0213 0132 2310} & \code{3+ 3- 0- 0- 3- 4+}\\
3 & + & \code{5 8 8 13} & \code{2103 2031 3012 3120} & \code{3+ 4+ 4- 3- 3+ 2+}\\
4 & + & \code{7 7 11 0} & \code{0213 1230 0321 0132} & \code{0+ 2+ 3+ 0+ 3+ 2+}\\
5 & + & \code{6 10 3 2} & \code{1302 3201 2103 0132} & \code{3+ 0- 2+ 3- 3+ 1+}\\
6 & + & \code{12 5 8 0} & \code{3201 2031 3120 0321} & \code{2+ 1+ 3+ 3- 3+ 3-}\\
7 & + & \code{4 10 4 12} & \code{0213 3012 3012 2031} & \code{3- 2+ 2- 0- 2+ 3+}\\
8 & + & \code{3 3 6 11} & \code{1302 1230 3120 1023} & \code{3- 3+ 3- 4+ 2- 4+}\\
9 & + & \code{13 1 0 10} & \code{3201 3012 0132 0321} & \code{0+ 1- 1+ 2- 3+ 3-}\\
10 & + & \code{7 9 5 1} & \code{1230 0321 2310 1302} & \code{1+ 1- 0+ 3+ 2- 2+}\\
11 & + & \code{13 2 4 8} & \code{2103 2031 0321 1023} & \code{3+ 4+ 0+ 3+ 3- 3+}\\
12 & + & \code{13 7 2 6} & \code{1302 1302 0213 2310} & \code{3- 3+ 0- 3- 4+ 2+}\\
13 & + & \code{3 12 11 9} & \code{3120 2031 2103 2310} & \code{3- 2+ 3+ 3- 3- 4-}\\
\end{longtable}
\endgroup

For each low-valence class $e$, the following records $(M,m,R)$ give the common adjacent minimum, opposite valence, and pairwise threshold, with multiplicities:
\begin{center}
\small
\begin{tabular}{cl}
\toprule
Class $e$ & Occurrence data\\
\midrule
4 & $4\times(8,38,38),\quad 1\times(12,38,12),\quad 2\times(12,38,38)$\\
 & $1\times(15,38,38)$\\
\bottomrule
\end{tabular}
\end{center}

\Needspace{7\baselineskip}
\subsection{\texorpdfstring{Data for $\mathcal P$}{Data for P}}\label{data:P}
This is the triangulation used in \Cref{ex:pair-witness}. Its edge valences, in label order, are $(8, 13, 13, 8, 8, 21, 8, 17)$.

\begingroup
\small
\setlength{\tabcolsep}{5pt}
\setlength{\LTpre}{6pt}
\setlength{\LTpost}{6pt}
\renewcommand{\LTcaptype}{} 
\stepcounter{LT@tables} 
\begin{longtable}{rclll}
\toprule
$t$ & $\varepsilon_t$ & Face targets & Face permutations & Signed local edge classes\\
\midrule
\endfirsthead
\multicolumn{5}{l}{\textit{Face-pairing data (continued)}}\\
\toprule
$t$ & $\varepsilon_t$ & Face targets & Face permutations & Signed local edge classes\\
\midrule
\endhead
\midrule
\multicolumn{5}{r}{\textit{Continued on next page}}\\
\endfoot
\bottomrule
\endlastfoot
0 & + & \code{2 5 1 7} & \code{0132 0132 0132 0132} & \code{0+ 1+ 2+ 3+ 4+ 5+}\\
1 & + & \code{5 11 2 0} & \code{0132 0213 0132 0132} & \code{0+ 2+ 2- 4+ 4+ 5+}\\
2 & + & \code{0 10 3 1} & \code{0132 0213 0132 0132} & \code{0+ 2- 1+ 4+ 3+ 5-}\\
3 & + & \code{15 7 4 2} & \code{3201 0132 0132 0132} & \code{0+ 1+ 1- 3+ 6+ 7+}\\
4 & + & \code{6 9 5 3} & \code{0132 3120 0132 0132} & \code{0+ 1- 2+ 6+ 4+ 5-}\\
5 & + & \code{1 0 6 4} & \code{0132 0132 0132 0132} & \code{0+ 2+ 1+ 4+ 4+ 5-}\\
6 & + & \code{4 10 7 5} & \code{0132 3120 0132 0132} & \code{0+ 1+ 1- 4+ 6+ 5+}\\
7 & + & \code{12 3 0 6} & \code{1230 0132 0132 0132} & \code{0+ 1- 1+ 6+ 3+ 7-}\\
8 & + & \code{9 12 14 15} & \code{3120 0213 0213 2031} & \code{3+ 7- 7+ 5- 7+ 5-}\\
9 & + & \code{10 4 11 8} & \code{0132 3120 2310 3120} & \code{7- 5+ 2- 5- 2+ 1+}\\
10 & + & \code{9 6 2 11} & \code{0132 3120 0213 1023} & \code{2- 5- 1+ 2+ 5- 1-}\\
11 & + & \code{13 9 1 10} & \code{3120 3201 0213 1023} & \code{2+ 2+ 2- 5- 5+ 7-}\\
12 & + & \code{13 7 8 14} & \code{2031 3012 0213 2103} & \code{7- 3- 7+ 5- 5- 6+}\\
13 & + & \code{14 15 12 11} & \code{2103 0213 1302 3120} & \code{5- 7+ 5- 5- 6- 7+}\\
14 & + & \code{15 8 13 12} & \code{0132 0213 2103 2103} & \code{5+ 3+ 7+ 7+ 6- 7+}\\
15 & + & \code{14 8 13 3} & \code{0132 1302 0213 2310} & \code{7+ 3- 5- 6- 7+ 7-}\\
\end{longtable}
\endgroup

For each low-valence class $e$, the following records $(M,m,R)$ give the common adjacent minimum, opposite valence, and pairwise threshold, with multiplicities:
\begin{center}
\small
\begin{tabular}{cl}
\toprule
Class $e$ & Occurrence data\\
\midrule
0 & $2\times(8,17,13),\quad 6\times(8,21,13)$\\
3 & $4\times(8,13,13),\quad 1\times(8,17,17),\quad 1\times(8,21,17)$\\
 & $1\times(17,8,17),\quad 1\times(17,21,17)$\\
4 & $8\times(8,13,13)$\\
6 & $4\times(8,13,13),\quad 1\times(8,17,21),\quad 1\times(8,21,17)$\\
 & $1\times(17,8,17),\quad 1\times(17,17,21)$\\
\bottomrule
\end{tabular}
\end{center}

\Needspace{7\baselineskip}
\subsection{\texorpdfstring{Data for $\mathcal U_6$}{Data for U6}}\label{data:U6}
This is the geometric triangulation of \Cref{ex:U6}. Its quotient-edge valences, in label order, are $(6, 17, 19)$, and its connected vertex link has genus 5. All tetrahedra have orientation sign $+1$.

\begingroup
\small
\setlength{\tabcolsep}{5pt}
\setlength{\LTpre}{6pt}
\setlength{\LTpost}{6pt}
\renewcommand{\LTcaptype}{}
\stepcounter{LT@tables}
\begin{longtable}{rclll}
\toprule
$t$ & $\varepsilon_t$ & Face targets & Face permutations & Signed local edge classes\\
\midrule
\endfirsthead
\multicolumn{5}{l}{\textit{Face-pairing data (continued)}}\\
\toprule
$t$ & $\varepsilon_t$ & Face targets & Face permutations & Signed local edge classes\\
\midrule
\endhead
\midrule
\multicolumn{5}{r}{\textit{Continued on next page}}\\
\endfoot
\bottomrule
\endlastfoot
0 & + & \code{5 4 1 5} & \code{1230 3012 0132 0132} & \code{0+ 1+ 2+ 1- 1- 1-}\\
1 & + & \code{5 3 2 0} & \code{0132 2031 0132 0132} & \code{0+ 2+ 2- 1- 1- 1+}\\
2 & + & \code{4 6 3 1} & \code{1302 1230 0132 0132} & \code{0+ 2- 2+ 1- 2+ 2-}\\
3 & + & \code{1 6 4 2} & \code{1302 2103 0132 0132} & \code{0+ 2+ 2- 2+ 1- 2+}\\
4 & + & \code{0 2 5 3} & \code{1230 2031 0132 0132} & \code{0+ 2- 1+ 1- 1- 2-}\\
5 & + & \code{1 0 0 4} & \code{0132 3012 0132 0132} & \code{0+ 1+ 1+ 1- 1- 1-}\\
6 & + & \code{6 3 2 6} & \code{3201 2103 3012 2310} & \code{2- 2- 2+ 2+ 2- 2-}\\
\end{longtable}
\endgroup

The low-valence occurrence records $(M,m,R)$, with multiplicities, are
\begin{center}
\small
\begin{tabular}{cl}
\toprule
Class $e$ & Occurrence data\\
\midrule
0 & $2\times(17,17,17),\quad 1\times(17,17,19)$\\
 & $1\times(17,19,17),\quad 2\times(17,19,19)$\\
\bottomrule
\end{tabular}
\end{center}

The low-edge occurrences are $T_0^{01},\ldots,T_{5}^{01}$. The face maps $T_i:F_2\to T_{i+1}:F_3$ with permutation \code{0132} (indices modulo 6) connect them in one cycle and preserve $0\to1$. All remaining faces omit either 0 or 1 and therefore cannot identify a new local edge with this class. The signs in the last column certify the orientations of the remaining quotient edges.

\Needspace{7\baselineskip}
\subsection{\texorpdfstring{Data for $\mathcal U_7$}{Data for U7}}\label{data:U7}
This is the geometric triangulation of \Cref{ex:U7}. Its quotient-edge valences, in label order, are $(7, 13, 13, 20, 13)$, and its connected vertex link has genus 7. All tetrahedra have orientation sign $+1$.

\begingroup
\small
\setlength{\tabcolsep}{5pt}
\setlength{\LTpre}{6pt}
\setlength{\LTpost}{6pt}
\renewcommand{\LTcaptype}{}
\stepcounter{LT@tables}
\begin{longtable}{rclll}
\toprule
$t$ & $\varepsilon_t$ & Face targets & Face permutations & Signed local edge classes\\
\midrule
\endfirsthead
\multicolumn{5}{l}{\textit{Face-pairing data (continued)}}\\
\toprule
$t$ & $\varepsilon_t$ & Face targets & Face permutations & Signed local edge classes\\
\midrule
\endhead
\midrule
\multicolumn{5}{r}{\textit{Continued on next page}}\\
\endfoot
\bottomrule
\endlastfoot
0 & + & \code{4 9 1 6} & \code{0132 3201 0132 0132} & \code{0+ 1+ 1- 2+ 2+ 3+}\\
1 & + & \code{10 9 2 0} & \code{3201 0132 0132 0132} & \code{0+ 1- 4+ 2+ 4+ 3+}\\
2 & + & \code{4 3 3 1} & \code{1023 1023 0132 0132} & \code{0+ 4+ 2+ 4+ 4+ 3-}\\
3 & + & \code{2 5 4 2} & \code{1023 0132 0132 0132} & \code{0+ 2+ 4+ 4+ 2+ 3-}\\
4 & + & \code{0 2 5 3} & \code{0132 1023 0132 0132} & \code{0+ 4+ 4+ 2+ 2+ 3-}\\
5 & + & \code{6 3 6 4} & \code{1023 0132 0132 0132} & \code{0+ 4+ 2+ 2+ 1+ 3+}\\
6 & + & \code{10 5 0 5} & \code{1230 1023 0132 0132} & \code{0+ 2+ 1+ 1+ 2+ 3+}\\
7 & + & \code{8 7 8 7} & \code{2103 2310 0213 3201} & \code{3- 3- 3+ 3+ 1+ 3+}\\
8 & + & \code{10 7 7 9} & \code{0213 0213 2103 1023} & \code{3- 3- 3+ 4+ 1+ 1+}\\
9 & + & \code{10 1 0 8} & \code{2103 0132 2310 1023} & \code{3+ 4+ 1- 3- 1+ 3-}\\
10 & + & \code{8 6 9 1} & \code{0213 3012 2103 2310} & \code{3+ 2- 3- 4- 1+ 1+}\\
\end{longtable}
\endgroup

The low-valence occurrence records $(M,m,R)$, with multiplicities, are
\begin{center}
\small
\begin{tabular}{cl}
\toprule
Class $e$ & Occurrence data\\
\midrule
0 & $7\times(13,20,13)$\\
\bottomrule
\end{tabular}
\end{center}

The low-edge occurrences are $T_0^{01},\ldots,T_{6}^{01}$. The face maps $T_i:F_2\to T_{i+1}:F_3$ with permutation \code{0132} (indices modulo 7) connect them in one cycle and preserve $0\to1$. All remaining faces omit either 0 or 1 and therefore cannot identify a new local edge with this class. The signs in the last column certify the orientations of the remaining quotient edges.

\Needspace{7\baselineskip}
\subsection{\texorpdfstring{Data for $\Pmix$}{Data for Pmix}}\label{data:Pmix}
This is the geometric triangulation of \Cref{ex:mixed-pair-eight}. Its quotient-edge valences, in label order, are $(8,13,13,8,8,13,8,25)$, and its connected vertex link has genus 9. All tetrahedra have orientation sign $+1$.

\begingroup
\small
\setlength{\tabcolsep}{5pt}
\setlength{\LTpre}{6pt}
\setlength{\LTpost}{6pt}
\renewcommand{\LTcaptype}{}
\stepcounter{LT@tables}
\begin{longtable}{rclll}
\toprule
$t$ & $\varepsilon_t$ & Face targets & Face permutations & Signed local edge classes\\
\midrule
\endfirsthead
\multicolumn{5}{l}{\textit{Face-pairing data (continued)}}\\
\toprule
$t$ & $\varepsilon_t$ & Face targets & Face permutations & Signed local edge classes\\
\midrule
\endhead
\midrule
\multicolumn{5}{r}{\textit{Continued on next page}}\\
\endfoot
\bottomrule
\endlastfoot
0 & + & \code{2 5 1 7} & \code{0132 0132 0132 0132} & \code{0+ 1+ 2+ 3+ 4+ 5+}\\
1 & + & \code{5 11 2 0} & \code{0132 0213 0132 0132} & \code{0+ 2+ 2- 4+ 4+ 5+}\\
2 & + & \code{0 10 3 1} & \code{0132 0213 0132 0132} & \code{0+ 2- 1+ 4+ 3+ 5-}\\
3 & + & \code{15 7 4 2} & \code{3201 0132 0132 0132} & \code{0+ 1+ 1- 3+ 6+ 7+}\\
4 & + & \code{6 9 5 3} & \code{0132 3120 0132 0132} & \code{0+ 1- 2+ 6+ 4+ 7-}\\
5 & + & \code{1 0 6 4} & \code{0132 0132 0132 0132} & \code{0+ 2+ 1+ 4+ 4+ 5-}\\
6 & + & \code{4 10 7 5} & \code{0132 3120 0132 0132} & \code{0+ 1+ 1- 4+ 6+ 7+}\\
7 & + & \code{12 3 0 6} & \code{1230 0132 0132 0132} & \code{0+ 1- 1+ 6+ 3+ 7-}\\
8 & + & \code{9 12 14 15} & \code{3120 0213 0213 2031} & \code{3+ 7- 7+ 5- 7+ 7-}\\
9 & + & \code{10 4 11 8} & \code{0132 3120 2310 3120} & \code{7- 7+ 2- 5- 2+ 1+}\\
10 & + & \code{9 6 2 11} & \code{0132 3120 0213 1023} & \code{2- 7- 1+ 2+ 5- 1-}\\
11 & + & \code{13 9 1 10} & \code{3201 3201 0213 1023} & \code{2+ 2+ 2- 7- 5+ 7-}\\
12 & + & \code{13 7 8 14} & \code{2031 3012 0213 2103} & \code{7- 3- 7+ 5- 7- 6+}\\
13 & + & \code{14 15 12 11} & \code{2103 0213 1302 2310} & \code{7- 7+ 5- 5- 6- 7+}\\
14 & + & \code{15 8 13 12} & \code{0132 0213 2103 2103} & \code{5+ 3+ 7+ 7+ 6- 7+}\\
15 & + & \code{14 8 13 3} & \code{0132 1302 0213 2310} & \code{7+ 3- 5- 6- 7+ 7-}\\
\end{longtable}
\endgroup

The low-valence occurrence records $(M,m,R)$, with multiplicities, are
\begin{center}
\small
\begin{tabular}{cl}
\toprule
Class $e$ & Occurrence data\\
\midrule
0 & $4\times(8,13,13),\quad4\times(8,25,13)$\\
3 & $4\times(8,13,13),\quad(8,25,13),\quad(8,25,25)$\\
  & $(13,8,25),\quad(13,25,25)$\\
4 & $8\times(8,13,13)$\\
6 & $4\times(8,13,13),\quad(8,13,25),\quad(8,25,25)$\\
  & $(13,8,25),\quad(13,25,13)$\\
\bottomrule
\end{tabular}
\end{center}
The signed edge data and reciprocal face maps satisfy the checks described at the start of this appendix. There are 8 circular edge links. The vertex link has 16 vertices, 96 edges, and 64 triangles, and the link of each of its vertices is a circle; it is therefore a closed orientable surface of Euler characteristic $-16$ and genus 9. The distinguished edge $0$ has occurrences $T_0^{01},\ldots,T_7^{01}$, with $m=13$ for $t=0,1,2,5$ and $m=25$ for $t=3,4,6,7$.

\section{Comparison of the tests for the mixed pair-min example}\label{app:mixed-comparison}
We justify the comparisons in \Cref{ex:mixed-pair-eight}. At the distinguished valence-8 edge of $\Pmix$, each occurrence has $(M,R)=(8,13)$, with four opposite valences 13 and four opposite valences 25.

For the fixed-parameter comparisons, the interval method of Appendix~\ref{app:numerics} gives
\begin{align*}
 \Theta_{C_0}(13,13)&\in[0.795016978000739048,\,0.795016978000739049],\\
 \Theta_{C_0}(13,25)&\in[0.783941662863045541,\,0.783941662863045542],\\
 4\Theta_{C_0}(13,13)+4\Theta_{C_0}(13,25)
 &\in[6.315834563455138356,\,6.315834563455138364].
\end{align*}
The second angle is less than $\pi/4$, but the sum is greater than $2\pi$. In the individual $C=2$ test, put
\[
 L_m=\arccos\varphi(2,b_{13},b_{13},\gamma_m,2,2).
\]
The sum at this edge is
\[
 4L_{13}+4L_{25}\in[6.267948814608128040,\,6.267948814608128048]
 \subset(-\infty,2\pi),
\]
so \Cref{thm:C2-star} does not apply.

For the symmetric tests over the full parameter range, take $m\ge13$ and $C>3/2$. Then $c_m\ge c_{13}>37/42$ and hence
\[
 \Gamma_m(C)=1+\frac{2(1-c_m)}{C+c_m}
 <1+\frac{2(1-37/42)}{3/2+37/42}=\frac{11}{10}.
\]
The bound on $c_{13}$ follows, for example, from $\cos t\ge1-t^2/2$ and $\pi<22/7$:
$c_{13}>7313/8281>37/42$. Also
\[
 \frac{C-1}{2C^2+C-1}\le\frac19,
 \qquad 2C^2+C-1-9(C-1)=2(C-2)^2.
\]
Thus \eqref{eq:symmetric-q}, with $p=C$, gives
\[
 \cos A_C(8,m)>1-\frac{1+11/10}{9}
 =\frac{23}{30}>\frac{\sqrt2}{2}.
\]
Every symmetric lower bound at this edge is below $\pi/4$, so no symmetric star sum succeeds for $C\in(3/2,1+\sqrt2]$. Moreover, $T_0$ contains this edge and has six distinct quotient-edge classes. Hence the local bichromatic hypothesis also fails.

We now show that the occurrence-wise pair-min test fails for the weaker data $(M,R,m)=(8,13,25)$ throughout the parameter range. Set
\[
 \lambda=\frac{2(1-c_{13})}{1+c_{13}}>0,
 \qquad u=C^2,\qquad \delta=\Gamma_{25}(C).
\]
Then $q=B_{13}(C)=1+\lambda C^2$, and direct simplification yields
\[
 \Phi_{\mathrm{same}}(C,q,\delta)
 =1-\frac{1+\delta}{\lambda+2}\frac{u-1}{u(2+\lambda u)}.
\]
First suppose $3/2<C\le2$. Put $h(u)=(u-1)/(u(2+\lambda u))$. The interval method of Appendix~\ref{app:numerics} certifies $3/25<\lambda<1/4$ and $\delta\le\Gamma_{25}(3/2)<21/20$. Since $h'(u)$ has the sign of $2+2\lambda u-\lambda u^2\ge2-8\lambda>0$ on $1<u\le4$, one has $h(u)\le h(4)$, and therefore
\[
 \Phi_{\mathrm{pair}}(C,q,\delta)
 \ge1-\frac{3(1+\delta)}{4(\lambda+2)(2+4\lambda)}
 >\frac{37201}{52576}>\frac{\sqrt2}{2},
\]
where the last comparison follows by squaring positive quantities.

Now suppose $2\le C\le1+\sqrt2$ and put $\gamma=\Gamma_{25}(2)$, so that $\delta\le\gamma$. The arithmetic--geometric mean inequality gives
\begin{align*}
 \frac{u(2+\lambda u)}{u-1}
 &=\lambda(u-1)+2\lambda+2+\frac{\lambda+2}{u-1}\\
 &\ge2\lambda+2+2\sqrt{\lambda(\lambda+2)}.
\end{align*}
Therefore
\[
 \Phi_{\mathrm{pair}}(C,B_{13}(C),\Gamma_{25}(C))\ge\kappa,
 \quad
 \kappa=1-\frac{1+\gamma}
 {(\lambda+2)\{2\lambda+2+2\sqrt{\lambda(\lambda+2)}\}}.
\]
Exact interval evaluation certifies
\[
 \kappa\in[0.707617762149055785,\,0.707617762149055786]
 \subset(\sqrt2/2,1).
\]
Both parameter ranges thus give $\Theta_C(13,25)<\pi/4$. Together with the symmetric comparison above, this proves
\[
 \max\{A_C(8,25),\Theta_C(13,25)\}<\pi/4
 \qquad(3/2<C\le1+\sqrt2).
\]
Thus summation across occurrences is essential for the stated symmetric/pair-min tests. These comparisons do not exclude stronger angle estimates or a direct verification of the full static-box condition.

\bibliographystyle{amsalpha}
\bibliography{references}

@article{BaoBonahon,
  author  = {Xiliang Bao and Francis Bonahon},
  title   = {Hyperideal polyhedra in hyperbolic 3-space},
  journal = {Bull. Soc. Math. France},
  volume  = {130},
  number  = {3},
  pages   = {457--491},
  year    = {2002},
  doi     = {10.24033/bsmf.2426},
  note = {\url{https://doi.org/10.24033/bsmf.2426}}
}

@article{CFMP,
  author  = {Francesco Costantino and Roberto Frigerio and Bruno Martelli and Carlo Petronio},
  title   = {Triangulations of 3-manifolds, hyperbolic relative handlebodies, and {D}ehn filling},
  journal = {Comment. Math. Helv.},
  volume  = {82},
  number  = {4},
  pages   = {903--933},
  year    = {2007},
  doi     = {10.4171/CMH/114},
  note = {\url{https://doi.org/10.4171/CMH/114}}
}

@article{FengGeHua,
  author  = {Ke Feng and Huabin Ge and Bobo Hua},
  title   = {Combinatorial {R}icci flows and the hyperbolization of a class of compact 3-manifolds},
  journal = {Geom. Topol.},
  volume  = {26},
  number  = {3},
  pages   = {1349--1384},
  year    = {2022},
  doi     = {10.2140/gt.2022.26.1349},
  eprint  = {2009.03731v2},
  eprinttype = {arxiv},
  note = {\url{https://doi.org/10.2140/gt.2022.26.1349}; \href{https://arxiv.org/abs/2009.03731v2}{arXiv:2009.03731v2}}
}

@article{HamPurcell,
  author  = {Sophie L. Ham and Jessica S. Purcell},
  title   = {Geometric triangulations and highly twisted links},
  journal = {Algebr. Geom. Topol.},
  volume  = {23},
  number  = {3},
  pages   = {1399--1462},
  year    = {2023},
  doi     = {10.2140/agt.2023.23.1399},
  note = {\url{https://doi.org/10.2140/agt.2023.23.1399}}
}

@article{HodgsonRubinstein,
  author  = {Craig D. Hodgson and J. Hyam Rubinstein and Henry Segerman},
  title   = {Triangulations of hyperbolic 3-manifolds admitting strict angle structures},
  journal = {J. Topol.},
  volume  = {5},
  number  = {4},
  pages   = {887--908},
  year    = {2012},
  doi     = {10.1112/jtopol/jts022},
  note = {\url{https://doi.org/10.1112/jtopol/jts022}}
}

@incollection{Kojima,
  author    = {Sadayoshi Kojima},
  title     = {Polyhedral decomposition of hyperbolic 3-manifolds with totally geodesic boundary},
  booktitle = {Aspects of Low Dimensional Manifolds},
  series    = {Adv. Stud. Pure Math.},
  volume    = {20},
  pages     = {93--112},
  publisher = {Kinokuniya},
  address   = {Tokyo},
  year      = {1992},
  doi       = {10.2969/aspm/02010093},
  note = {\url{https://doi.org/10.2969/aspm/02010093}}
}

@article{Luo2005,
  author  = {Feng Luo},
  title   = {A combinatorial curvature flow for compact 3-manifolds with boundary},
  journal = {Electron. Res. Announc. Amer. Math. Soc.},
  volume  = {11},
  pages   = {12--20},
  year    = {2005},
  doi     = {10.1090/S1079-6762-05-00142-3},
  note = {\url{https://doi.org/10.1090/S1079-6762-05-00142-3}}
}

@article{LuoYang,
  author  = {Feng Luo and Tian Yang},
  title   = {Volume and rigidity of hyperbolic polyhedral 3-manifolds},
  journal = {J. Topol.},
  volume  = {11},
  number  = {1},
  pages   = {1--29},
  year    = {2018},
  doi     = {10.1112/topo.12046},
  note = {\url{https://doi.org/10.1112/topo.12046}}
}

@misc{ThurstonNotes,
  author       = {William P. Thurston},
  title        = {The Geometry and Topology of Three-Manifolds},
  year         = {1980},
  howpublished = {Princeton lecture notes},
  note         = {\url{https://library.slmath.org/books/gt3m/}}
}

@misc{Zhao,
  author       = {Xinrong Zhao},
  title        = {Combinatorial {R}icci flows and hyperbolic structures on a class of compact 3-manifolds with boundary},
  year         = {2026},
  howpublished = {arXiv:2601.15174v3},
  note         = {\url{https://arxiv.org/abs/2601.15174v3}}
}

@article{PaoluzziZimmermann,
  author  = {Luisa Paoluzzi and Bruno Zimmermann},
  title   = {On a class of hyperbolic 3-manifolds and groups with one defining relation},
  journal = {Geom. Dedicata},
  volume  = {60},
  pages   = {113--123},
  year    = {1996},
  doi     = {10.1007/BF00160617},
  note = {\url{https://doi.org/10.1007/BF00160617}}
}

@article{FominykhVesnin,
  author        = {A. Yu. Vesnin and E. A. Fominykh},
  title         = {Exact values of complexity for {P}aoluzzi--{Z}immermann manifolds},
  journal       = {Dokl. Math.},
  volume        = {84},
  number        = {1},
  pages         = {542--544},
  year          = {2011},
  eprint        = {1105.2542},
  archivePrefix = {arXiv},
  primaryClass  = {math.GT},
  doi           = {10.1134/S1064562411050139},
  note = {\url{https://doi.org/10.1134/S1064562411050139}; \href{https://arxiv.org/abs/1105.2542}{arXiv:1105.2542}}
}

@article{FrigerioPetronio,
  author = {Roberto Frigerio and Carlo Petronio},
  title = {Construction and recognition of hyperbolic 3-manifolds with geodesic boundary},
  journal = {Trans. Amer. Math. Soc.},
  volume = {356}, number = {8}, pages = {3243--3282}, year = {2004},
  doi = {10.1090/S0002-9947-03-03378-6},
  note = {\url{https://doi.org/10.1090/S0002-9947-03-03378-6}}
}

\end{document}